\documentclass[a4paper,11pt,reqno]{amsart}
\usepackage{amsmath,a4wide}
\usepackage{xfrac}
\usepackage{stmaryrd,mathrsfs,bm,amsthm,mathtools,yfonts,amssymb,color}
\usepackage{esint}
\usepackage{xcolor}
\usepackage[pagebackref,citecolor=black,linkcolor=blue,colorlinks=true]{hyperref}
\usepackage{tikz,tikz-3dplot}

\begingroup
\newtheorem{theorem}{Theorem}[section]
\newtheorem{lemma}[theorem]{Lemma}
\newtheorem{proposition}[theorem]{Proposition}
\newtheorem{corollary}[theorem]{Corollary}
\endgroup

\theoremstyle{definition}
\newtheorem{definition}[theorem]{Definition}
\newtheorem{remark}[theorem]{Remark}
\newtheorem{ipotesi}[theorem]{Assumption}

\numberwithin{equation}{section}

\newcommand{\R}{\mathbb{R}} 

\newcommand{\bA}{\mathbf{A}} 
\newcommand{\bC}{\mathbf{C}} 
\newcommand{\bB}{\mathbf{B}} 
\newcommand{\bS}{\mathbf{S}} 
\newcommand{\bOmega}{\mathbf{\Omega}}
\newcommand{\Om}{\Omega}

\newcommand{\p}{\mathbf{p}} 

\newcommand{\G}{\mathcal{G}} 
\newcommand{\bfeta}{\boldsymbol{\eta}} 
\newcommand{\Dir}{\mathrm{Dir}} 
\newcommand{\bG}{\mathbf{G}} 

\newcommand{\mass}{\mathbf{M}} 
\newcommand{\Rc}{\mathscr{R}} 
\newcommand\res{\mathop{\hbox{\vrule height 7pt width .3pt depth 0pt\vrule height .3pt width 5pt depth 0pt}}\nolimits}

\newcommand{\reg}{\mathrm{Reg}} 
\newcommand{\sing}{\mathrm{Sing}} 
\newcommand{\Sing}{\mathrm{Sing}} 

\newcommand{\bE}{\mathbf{E}} 
\newcommand{\bh}{\mathbf{h}} 
\newcommand{\bD}{\mathbf{D}}
\newcommand{\bSigma}{\mathbf{\Sigma}}
\newcommand{\bI}{\mathbf{I}}

\newcommand{\modp}{{\rm mod}(p)} 
\newcommand{\moddue}{{\rm mod}(2)} 

\newcommand{\Ha}{\mathcal{H}} 

\newcommand{\eps}{\varepsilon} 
\newcommand{\spt}{\mathrm{spt}} 
\newcommand{\dist}{\mathrm{dist}} 
\newcommand{\B}{\mathbf{B}} 
\newcommand{\Lip}{\mathrm{Lip}} 
\renewcommand{\epsilon}{\varepsilon}

\def\XXint#1#2#3{{\setbox0=\hbox{$#1{#2#3}{\int}$ }
		\vcenter{\hbox{$#2#3$ }}\kern-.6\wd0}}

\newcommand{\mres}{\mathbin{\vrule height 1.6ex depth 0pt width 
		0.13ex\vrule height 0.13ex depth 0pt width 1.3ex}}

\def\a#1{\left\llbracket{#1}\right\rrbracket}
\newcommand{\abs}[1]{\lvert#1\rvert} 
\newcommand{\etab}{\boldsymbol{\eta}}

\newcommand{\Iqspecl}{{\mathscr{A}_Q (\R)}}

\newcommand{\bH}{\mathbf{H}}

\newcommand{\be}{\mathbf{e}}

\newcommand\bmo{{\bm m}_0}
\newcommand\sC{{\mathscr{C}}}

\newcommand\sW{{\mathscr{W}}}

\newcommand\cM{{\mathcal{M}}}

\newcommand\cG{{\mathcal{G}}}

\title[Area minimizing hypersurfaces modulo $p$]{Fine structure of the singular set \\ of area minimizing hypersurfaces modulo $p$}

\author[C. De Lellis]{Camillo De Lellis}
\address{School of Mathematics, Institute for Advanced Study, 1 Einstein Dr., Princeton NJ 05840, USA}
\email{camillo.delellis@math.ias.edu}
\author[J. Hirsch]{Jonas Hirsch}
\address{Mathematisches Institut, Universit\"at Leipzig, Augustusplatz 10, D-04109 Leipzig, Germany}
\email{hirsch@math.uni-leipzig.de}
\author[A. Marchese]{Andrea Marchese}
\address{Dipartimento di Matematica, Universit\`a degli Studi di Trento, Via Sommarive 14, I-38123 Povo (TN), Italy}
\email{andrea.marchese@unitn.it}
\author[L. Spolaor]{Luca Spolaor}
\address{Department of Mathematics, UC San Diego, AP\&M, La Jolla, California, 92093, USA}
\email{lspolaor@ucsd.edu}
\author[S. Stuvard]{Salvatore Stuvard}
\address{Dipartimento di Matematica, Universit\`a degli Studi di Milano, Via Saldini 50, I-20133 Milano (MI), Italy}
\email{salvatore.stuvard@unimi.it}

\begin{document}

\maketitle

\begin{abstract}
    Consider an area minimizing current modulo $p$ of dimension $m$ in a smooth Riemannian manifold of dimension $m+1$. We prove that its interior singular set is, up to a relatively closed set of dimension at most $m-2$, a $C^{1,\alpha}$ submanifold of dimension $m-1$ at which, locally, $N\leq p$ regular sheets of the current join transversally, each sheet counted with a positive multiplicity $k_i$ so that $\sum_i k_i = p$. This completes the analysis of the structure of the singular set of area minimizing hypersurfaces modulo $p$, initiated by J. Taylor for $m=2$ and $p=3$ and extended by the authors to arbitrary $m$ and all \emph{odd} $p$. We tackle the remaining case of even $p$ by showing that the set of singular points admitting a flat blow-up is of codimension at least two in the current. First, we prove a structural result for the singularities of minimizers in the linearized problem, by combining an epiperimetric inequality with an analysis of homogeneous minimizers to conclude that the corresponding degrees of homogeneity are always integers; second, we refine Almgren's blow-up procedure to prove that \emph{all} flat singularities of the current persist as singularities of the $\Dir$-minimizing limit. An important ingredient of our analysis is the uniqueness of flat tangent cones at singular points, recently established by Minter and Wickramasekera in \cite{MW}.
\end{abstract}

\tableofcontents

\section{Introduction} \label{sec:intro}

This paper completes the analysis of the structure of the singular set of codimension one area minimizing currents modulo $p$, where $p \ge 2$ is an arbitrary integer, initiated by J. Taylor in \cite{Taylor}. Currents modulo $p$ were introduced by Ziemer for $p=2$ in \cite{Ziemer} and Federer for arbitrary $p$ in \cite{Federer69} in order to recast Plateau's problem in a class of generalized surfaces capable of exhibiting singularities which are not allowed in the class of area minimizing integral currents. For the terminology related to currents and currents modulo $p$ we refer to \cite{Federer69} and \cite{DLHMS}, and we give the following  

\begin{definition} \label{def:am_modp}
	Let $p\geq 2$, $\Omega \subset \R^{m+n}$ be open, and let $\Sigma \subset \R^{m+n}$ be a complete submanifold without boundary of dimension $m+\bar{n}$ and class $C^{3, \alpha}$ for some positive $\alpha$. We say that an $m$-dimensional integer rectifiable current $T \in \Rc_{m}(\Sigma)$ is \emph{area minimizing} $\modp$ in $\Sigma \cap \Omega$ if
	\begin{equation}\label{e:am_mod_p}
	\mass (T) \leq \mass (T + W) \qquad \mbox{for any $W \in \Rc_{m}(\Omega\cap \Sigma)$ which is a boundary $\modp$}.
	\end{equation}
\end{definition}

The set of interior regular points, denoted by $\reg (T)$, is the relatively open set of points $x\in \spt^p (T)$ in a neighborhood of which $T$ can be represented by a regular oriented submanifold of $\Sigma$ with constant multiplicity, cf. \cite[Definition 1.3]{DLHMS}. Its ``complement'', i.e.
\begin{equation} \label{e:singular set}
\sing (T) := (\Omega \cap  \spt^p (T)) \setminus (\reg (T) \cup \spt^p (\partial T))\, ,
\end{equation}
is the set of \emph{interior singular points}. It was recently established in \cite{DLHMS} that the Hausdorff dimension of $\Sing (T)$ cannot exceed $m-1$, the estimate being optimal, and that it is countably $(m-1)$-rectifiable with locally finite $\Ha^{m-1}$ measure whenever $p$ is odd. In this paper we then focus on the fine structure of $\Sing(T)$, and from now on we restrict our attention to the case when the codimension is $\bar n=1$. Prior to the present paper, the state of the art concerning the structural properties of the singular set for codimension 1 area minimizing currents mod $p$ was as follows:
\begin{itemize}
\item[(a)] When $p=2$, $\mathcal{H}^{m-2} (\sing (T))=0$ even in the case of minimizers of general uniformly elliptic integrands, see \cite{ASS}; for the area functional, using \cite{NV}, one can conclude additionally that $\sing (T)$ is countably $(m-7)$-rectifiable and has locally finite $\mathcal{H}^{m-7}$ measure.
\item[(b)] When $p=3$ and $m=2$, \cite{Taylor} gives a complete description of $\sing (T)$: the latter is shown to consist of $C^{1,\alpha}$ arcs where three regular sheets meet with equal angles; for higher $m$ the same structural result holds outside of a closed subset of dimension at most $m-3$, cf. \cite{Simon95}. 
\item[(c)] When $p=4$, \cite{White79} shows that minimizers of uniformly elliptic integrands are represented by {\em immersed manifolds} outside of a closed set of zero $\mathcal{H}^{m-2}$ measure.
\item[(d)] When $p$ is odd, the authors proved recently in \cite{DLHMSS} that $\sing (T)$ is a $C^{1,\alpha}$ submanifold of $\Sigma$ of dimension $m-1$ outside of a relatively closed set which is countably $(m-2)$-rectifiable and of locally finite $\Ha^{m-2}$ measure. It was then pointed out by Minter and Wickramasekera in \cite{MW} that the same conclusion can be inferred from the regularity theory for stable hypervarifolds developed by Wickramasekera in \cite{Wic}.
\end{itemize}

Better regularity properties can be inferred for every $p$ under suitable topological assumptions on the boundary $\partial^p[T]$, see e.g. \cite{Morgan_modp}. The regularity of the singular set established in point (d) induces a rigid local structure of the current itself, described in the following definition.

\begin{definition}\label{def:free-boundary}
Given an open set $U$, we say that $\sing (T)\cap U$ is a \emph{classical free boundary} if the following holds for some positive $\alpha$. 
\begin{itemize}
    \item[(i)] $\sing (T)\cap U$ is an orientable $C^{1,\alpha}$ $(m-1)$-dimensional submanifold of $U\cap \Sigma$; 
    \item[(ii)] $\reg (T) \cap U$ consists of $N\leq p$ connected $C^{1,\alpha}$ orientable submanifolds $\Gamma_i$ with $C^{1,\alpha}$ boundary $\partial \Gamma_i \cap U = \sing (T)\cap U$;
    \item[(iii)] There are $k_i\in \{1, \ldots , \lfloor \frac{p-1}{2}\rfloor\}$ such that, after suitably orienting $\sing (T) \cap U$ and $\Gamma_i$, 
\begin{align*}
&S := \sum_i k_i \a{\Gamma_i} \equiv T\res U\, \modp\\
&\partial S \res U= \sum_i k_i \a{\sing (T)\cap U} = p \a{\sing (T)\cap U}\, .
\end{align*}
\end{itemize}
A set $A \subset \sing(T)$ is locally a classical free boundary if for every $q \in A$ there is an open neighborhood $U \ni q$ such that $\sing (T) \cap U$ is a classical free boundary.
\end{definition}

The first main result of the present paper is the following.

\begin{theorem}\label{t:even}
Let $p= 2Q$ be even, and let $\Sigma, T$, and $\Omega$ be as in Definition \ref{def:am_modp}. If $\dim (\Sigma)= m+1$, then $\sing (T)$ is locally a classical free boundary outside of a relatively closed set $\mathcal{S}$ which has Hausdorff dimension at most $m-2$, and which is countable when $m=2$.
\end{theorem}

In view of (d) the theorem above completes the description of the top-dimensional part of the singular set for any $p$. We also use it to obtain a structural result for $T$, which extends to the case of even $p$ a similar statement proved in \cite{DLHMS} for odd $p$. Note, however, that the proof given in \cite{DLHMS} for $p$ odd is valid in any codimension, while our current knowledge for even $p$'s is restricted to codimension $1$. Moreover, using the strong maximum principle in \cite{Wic2}, the regularity theory for area-minimizing integral hypercurrents, and the structure theorem of \cite{NV}, we can take advantage of the codimension $1$ assumption to derive a more detailed description of $T$ compared to the one given in \cite{DLHMS}.  

\begin{theorem}\label{t:even-structure}
Let $p= 2Q$ be even, and let $\Sigma, T$, and $\Omega$ be as in Definition \ref{def:am_modp}. Assume in addition that $\dim (\Sigma)= m+1$ and $\partial^p [T] = 0$ in $\Omega$. Then we can decompose 
\[
T=T_o + T_n \;\modp
\]
for two currents $T_o$ and $T_n$ with the following properties:
\begin{itemize}
\item[(i)] $T_o$ and $T_n$ are both area minimizing $\modp$, $\mass^p (T_o) + \mass^p (T_n) = \mass^p (T)$, and in fact $\spt^p (T_o)\cap \spt^p (T_n) \cap \Omega = \emptyset$; in particular, $\Sing(T) = \Sing(T_o) \cup \Sing(T_n)$.
\item[(ii)] There is an integer rectifiable current $T_1$ with $\spt (\partial T_1)\cap\Omega\subset \sing (T)$ such that $T_o = T_1\; \modp$ and, after endowing $\sing (T)\setminus \mathcal{S}$ with a suitable orientation, $\partial T_1 = p \a{\sing (T)}$ in $\Omega$.  
\item[(iii)] There is an integer rectifiable current $T_2$ which is representative $\moddue$, area minimizing $\moddue$, satisfies $\partial^2 [T_2] = 0$ in $\Omega$ and is such that $T_n = Q T_2\; \modp$.
 \item[(iv)] For every $q\in \Sigma \cap \Omega$ there is ball $\bB_\rho (q)\subset \Omega$ such that $T_2\res \bB_\rho (q)$ has a $\moddue$ representative which is an area-minimizing integral current in $\Sigma \cap \bB_\rho (q)$ with no boundary in $\Sigma \cap \bB_\rho (q)$. In particular, $\sing (T_n) = \sing (T_2)$ is $(m-7)$-rectifiable and has locally finite $\mathcal{H}^{m-7}$-measure.
\end{itemize}
\end{theorem}

The next Section \ref{s:synopsis} gives a complete account of the strategy of proof of Theorems \ref{t:even} and \ref{t:even-structure}, which will then be developed in details in the rest of the paper. A fundamental step, which is also of independent interest, is the proof of a structural theorem on the singular set of solutions to the linearized problem corresponding to the minimization of the mass modulo $p$, that is the minimization of a suitably defined Dirichlet energy on a class of special multiple valued functions; see Theorem \ref{thm:linear-what-is-needed} below.

It is worth mentioning that such linearized problem admits a natural interpretation as a Free-Boundary problem; in fact, the techniques employed to prove Theorem \ref{thm:linear-what-is-needed} can be adapted to provide an alternative, purely variational, proof of the profound results in \cite{CaLi}. For more on this connection, we refer the reader to Remark \ref{rmk FB}.

\medskip

\noindent\textbf{Acknowledgements.} The authors thank Frank Morgan for comments which lead to an improved version of Theorem \ref{t:even-structure}. C.D.L. acknowledges support from the National Science Foundation through the grant FRG-1854147. J.H. was partially supported by the German Science Foundation DFG in context of the Priority Program SPP 2026 “Geometry at Infinity”. L.S. acknowledges the support of the NSF grant DMS-2044954.

\section{Synopsis} \label{s:synopsis}

Following \cite[Section 3]{DLHMSS}, we denote by $\mathcal S^k$ the \emph{stratum} of the support of $T$ consisting of all points $q \in \spt^p(T)\setminus \spt^p(\partial T)$ such that no tangent cones to $T$ at $q$ possess $k+1$ independent symmetries, and we recall the following facts:
\begin{itemize}
    \item From Almgren's stratification theorem it follows that $\mathcal{S}^{m-2}$ has Hausdorff dimension at most $m-2$;
    \item From \cite[Corollary 3.1 and Theorem 3.2]{DLHMSS} it follows that $\mathcal{S}^{m-1}\setminus \mathcal{S}^{m-2}$ is locally a classical free boundary (and thus it is relatively open).
\end{itemize}
When $p=2$ or $p$ is odd, points in $\mathcal S^m \setminus \mathcal S^{m-1}$ are necessarily regular: see \cite{Allard72} and \cite{White86}, respectively. This is false when $p = 2Q \ge 4$, as shown in \cite[Example 1.6]{DLHMS}. We thus recall the notation ${\rm Sing}_f (T)$ for all ``flat singular points'', namely those singular points $q\in {\rm spt}\, (T)\setminus {\rm spt}^p\, (\partial T)$ which have at least one \emph{flat} tangent cone. As a consequence of White's regularity theorem in \cite{White86}, any flat tangent cone at a singular point must have multiplicity $Q$ modulo $p$, and thus all points $q \in \sing_f (T)$ have density $\Theta_T(q) = Q$. The two bullets above reduce Theorem \ref{t:even} to the following

\begin{theorem}\label{t:even-reduced}
Let $T$ be as in Theorem \ref{t:even}. Then ${\rm Sing}_f (T)$ has Hausdorff dimension at most $m-2$, and it is countable when $m=2$.
\end{theorem}

In \cite[Theorem A]{MW}, Minter and Wickramasekera showed that, when $\Sigma=\R^{m+1}$, the flat tangent cone to $T$ at any point $q\in {\rm Sing}_f (T)$ is unique, and the rescaled currents $T_{q,r}$ converge to it with a power-law decay rate. The proof in \cite{MW} is based on the regularity theory of Wickramasekera for stable hypervarifolds \cite{Wic} and in particular the validity of the uniqueness of tangent cones and the power-law decay rate can be extended to the case in which $\Sigma$ is a more general Riemannian manifold along the lines explained in \cite{Wic}. This uniqueness and power-law decay rate is an essential starting point for our proof of Theorem \ref{t:even-reduced}. The specific form needed by our arguments is given in
Theorem \ref{t:uniqueness-tangent-plane}, because it differs slightly from the statements in \cite{MW}. We will assume it as a starting point of our analysis, and we shall present a self-contained proof, alternative to that of \cite{MW} and which uses directly the $\modp$ minimizing property, in the forthcoming work \cite{DLHMSS2}. 

We denote by $\bE^{no}$ the unoriented excess used in \cite{DLHMS} (which coincides with Allard's varifold excess, cf. \cite{Allard72}) and by $\bA$ the $L^\infty$ norm of the second fundamental form of the ambient manifold $\Sigma$. In fact, the unoriented excess $\bE^{no}$ used in \cite{DLHMS} is defined in cylinders and, for the reader's convenience, we specify its obvious extension to balls. First of all, given two $m$-dimensional planes $\pi$ and $\pi'$, we let $|\pi'-\pi|_{no}$ be the Hilbert-Schmidt norm of the difference between the orthogonal projections $\mathbf{p}_{\pi'}$ and $\mathbf{p}_\pi$ onto the corresponding planes. Hence, given a plane $\pi$ we define the spherical unoriented excess of $T$ with respect to $\pi$ in $\bB_r (q)$ as the quantity
\[
\bE^{no} (T, \bB_r (q), \pi) = \frac{1}{2\omega_m r^m} \int_{\bB_r (q)} |\pi (x) - \pi|_{no}^2 d\|T\| (x) \, ,
\]
where $\pi (x)$ is the plane oriented by the $m$-vector $\vec{T} (x)$. Finally, the spherical unoriented excess of $T$ in $\bB_r (q)$ is given by 
\[
\bE^{no}(T,\bB_r(q)):=\min_\pi \bE^{no} (T, \bB_r (q), \pi)\, .  
\]
We also recall that the height of the current $T$ in a set $E$ with respect to a given plane $\pi$ is defined as 
\[
\bh (T, E, \pi) := \sup_{q_1,q_2\in E\cap \spt^p (T)} |\mathbf{p}_{\pi^\perp} (q_1-q_2)|\, . 
\]
We are now ready to state the uniqueness of flat tangent cones and the corresponding decay rate which will be taken as starting point of our analysis.

\begin{theorem}\label{t:uniqueness-tangent-plane}
There are dimensional constants $\varepsilon_0 (p, m)>0$, $\alpha (p,m)>0$, and $C (p,m) > 0$ with the following property. Assume $T$ is as in Theorem \ref{t:even}, $q\in {\rm Sing}_f (T)$, $\bB_\rho (q)\cap {\rm spt}^p\, (\partial T) = \emptyset$ and 
\begin{equation}\label{e:smallness}
\bE^{no} (T, \bB_\rho (q)) + \rho^2 \bA^2 < \varepsilon_0\, .    
\end{equation}
Then, there is a unique tangent cone to $T$ at $q$: it has the form $Q \a{\pi (q)}$ for some $m$-dimensional plane $\pi(q)$, and moreover the following estimates hold
\begin{align}
\bE^{no} (T, \bB_r (q), \pi (q)) \leq C \frac{r^\alpha}{\rho^\alpha} (\bE^{no} (T, \bB_\rho (q)) + \rho^2 \bA^2)\, ,\label{e:decay}\\\label{e:starting_height_decay}
\bh (T, \bB_r (q), \pi (q)) \leq C r \frac{r^{\alpha/2}}{\rho^{\alpha/2}} (\bE^{no} (T, \bB_\rho (q)) + \rho^2 \bA^2)^{1/2}\, .
\end{align}
\end{theorem}

The first step of our proof will be an improvement of the power-law decay in Theorem \ref{t:uniqueness-tangent-plane} to an ``almost quadratic decay''. The key for this improvement is a classification of the possible frequency values for Dir-minimizing special $Q$-valued functions, see Theorem \ref{thm:linear-what-is-needed} below. Special $Q$-valued functions were introduced in \cite{DLHMS_linear} as a toolbox to linearize the area functional precisely at flat singular points of density $Q=\frac{p}{2}$. In the statement of Proposition \ref{p:decay-improved}, we use the notation $\bC_r (q,\pi)$ for the cylinder $\{x \, \colon \,\abs{\mathbf{p}_\pi (x - q)} < r \}$, and $\bE^{no}(T,\bC_r(q,\pi))$ for the cylindrical unoriented excess of $T$
\[
\bE^{no}(T,\bC_r(q,\pi)) = \frac{1}{2 \omega_m r^m} \int_{\bC_r (q,\pi)} \abs{\pi(x) - \pi}^2_{no} \, d\|T\|(x)\,.
\]

\begin{proposition}[Almost quadratic excess decay]\label{p:decay-improved}
For every $\delta>0$ there are constants $\varepsilon_1 (\delta,p, m)>0$ and $C (\delta, p,m)>0$ with the following property. Assume $T$ is as in Theorem \ref{t:even}, $q\in {\rm Sing}_f (T)$, $\bB_{4\rho} (q)\cap {\rm spt}^p\, (\partial T) = \emptyset$ and 
\begin{equation}\label{e:smallness-2}
\bE^{no} (T, \bB_{4\rho} (q), \pi (q)) + (4\rho)^2 \bA^2 < \varepsilon_1\, .    
\end{equation}
Assume in addition that $\bC_{\rho} (q, \pi (q))\cap \spt^p (T)\subset \bB_{2\rho} (q)$. Then for every $0 < r \leq \frac{\rho}{8}$ we have
\begin{align}
&\bE^{no} (T, \bB_r (q), \pi (q)) \leq \bE^{no} (T, \bC_r (q, \pi (q))) \leq C \frac{r^{2-2\delta}}{\rho^{2-2\delta}} (\bE^{no} (T, \bB_\rho (q), \pi (q)) + \rho^2 \bA^2)\label{e:decay-quadratic}\, ,\\ \label{e:decay-height}
&\bh (T, \bB_r (q), \pi (q)) \leq C r \frac{r^{1-\delta}}{\rho^{1-\delta}} (\bE^{no} (T, \bB_\rho (q)) + \rho^2 \bA^2)^{1/2}\,.
\end{align}
\end{proposition}

For the second step we need the center manifold introduced in \cite[Sections 17.1 and 17.2]{DLHMS} and we refer to \cite{DLHMS} for all the relevant definitions pertaining to it. The center manifold is a $C^3$ submanifold which approximates the current rather efficiently in a ball where the spherical excess is sufficiently small and whose construction, which follows a Whitney-type (or Calder\'on-Zygmund-type) decomposition of the space, depends on certain numerical constants specified in \cite[Assumption 17.11]{DLHMS}.

The next proposition states that an appropriate choice of these parameters guarantees that \emph{all} the flat singular points of $T$ lie in the center manifold. In fact we need a more precise statement which estimates the size of the cubes in the Calder\'on-Zygmund-type decomposition leading to the construction of the center manifold which are close to a flat singular point. 

\begin{proposition}[Enhanced Center Manifold]\label{p:contact} There is a constant $\eta (p, m)>0$ with the following property. 
Let $T$ as in Theorem \ref{t:even-reduced} satisfy that $0 \in \Sing_f (T)$ as well as, in addition, all the assumptions of \cite[Assumption 17.5]{DLHMS}. If the parameters in \cite[Assumption 17.11]{DLHMS} are chosen appropriately (in particular as it is detailed in Assumption \ref{a:parameters_fine_cm}), then 
\begin{equation}\label{e:fine_cm1}
{\rm Sing}_f (T)\cap \bB_\eta \subset \Phi (\Gamma)\,,
\end{equation}
i.e. the contact set $\Phi(\Gamma)$ of the center manifold $\mathcal{M}$ contains \emph{all} the flat singular points of $T$ in $\bB_\eta$. Moreover for every $q\in {\rm Sing}_f (T)\cap \bB_\eta$, writing $q=\left( x_q, y_q \right) \in \pi_0 \times \pi_0^\perp$ (where $\pi_0$ is the $m$-plane appearing in \cite[Assumption 17.5]{DLHMS}) we have
\begin{equation}\label{e:fine_cm2}
L\in \sW \qquad \Rightarrow \qquad \ell(L)<\frac{1}{64\sqrt{m}}\, \dist(x_q,L)\,,
\end{equation}
that is all cubes where the refinement stops close to a flat singularity have small side-length.
\end{proposition}

Note next that, having fixed any point $q\in {\rm Sing}_f (T)$, the current $T_{q,r}:= (\iota_{q,r})_\sharp T$ (where $\iota_{q,r} (\bar q) := \frac{\bar q-q}{r}$) falls under the assumptions of Proposition \ref{p:contact}, provided $r$ is sufficiently small. Therefore a standard covering argument and Proposition \ref{p:contact} reduce the proof of Theorem \ref{t:even-reduced} to the following statement. 

\begin{theorem}\label{t:even-reduced-2}
Let $T$ be as in Proposition \ref{p:contact} and $\mathcal{M}$ be the corresponding center manifold. Then, ${\rm Sing}_f (T) \cap\overline{\mathbf{B}}_1\cap \mathcal{M}$ has Hausdorff dimension at most $m-2$, and it is countable when $m=2$.
\end{theorem}

The third step is a suitable almost monotonicity property of the frequency function of the normal approximation $N$ of $T$ over the center manifold $\mathcal{M}$. The normal approximation is borrowed from \cite[Section 17.3]{DLHMS}. In order to define the frequency function we introduce the following Lipschitz (piecewise linear, radial) weight 
\begin{equation*}
\phi (r) :=
\begin{cases}
1 & \text{for }\, r\in [0,\textstyle{\frac{1}{2}}],\\
2-2r & \text{for }\, r\in \,\, ]\textstyle{\frac{1}{2}},1],\\
0 & \text{for }\, r\in \,\, ]1,+\infty[\, ,
\end{cases}
\end{equation*}
and denote by $d (x,y)$ the geodesic distance between $x,y\in \mathcal{M}$. 
\begin{definition}[Frequency function]\label{d:frequency}
For every $r\in ]0,1[$ and every $q\in \mathcal{M}$ 
\[
\mathbf{D} (q, r) := \int_{\mathcal{M}} \phi\left(\frac{d (x,q)}{r}
\right)\,|D N|^2(x)\, dx\quad\mbox{and}\quad
\mathbf{H} (q,r) := - \int_{\mathcal{M}} \phi'\left(\frac{d (x,q)}{r}\right)\,\frac{|N|^2(x)}{d(x,q)}\, dx\, .
\]
If $\mathbf{H} (q, r) > 0$, we define the {\em frequency function}
$\mathbf{I} (q,r) := \frac{r\,\mathbf{D} (q,r)}{\mathbf{H} (q,r)}$. 
\end{definition}

The relevant conclusion is the following almost monotonicity formula for points $q$ in the set ${\rm Sing}_f (T) \cap\overline{\mathbf{B}}_1\cap \mathcal{M}$.

\begin{proposition}[Almost monotonicity of the frequency function]\label{p:almost-monot}
Let $T$ be as in Theorem \ref{t:even-reduced-2}. There exist $0 < \bar \eta < \eta(p,m)$ and $\bar r > 0$ such that the frequency function $\mathbf{I} (q, r)$ is well defined for every $(q,r)\in \left( {\rm Sing}_f (T) \cap\overline{\mathbf{B}}_{\bar\eta} \right) \times ]0,\bar r]$.
Moreover, there exist functions $\Lambda=\Lambda(q,r)$ and $\Xi = \Xi(q,r)$ on $\left( {\rm Sing}_f (T) \cap\overline{\mathbf{B}}_{\bar\eta} \right) \times ]0,\bar r]$ such that $0<\Lambda(q,r) \leq C\,r^\gamma$ and $0<\Xi(q,r) \leq C\, \bD^\gamma(q,r)$ for some $\gamma > 0$ and $C > 0$  and, moreover:
\begin{equation}\label{e:almost-monot}
\frac{d}{dr} \Big(\exp (\Lambda(q,r)) \mathbf{I} (q,r) + \Xi(q,r)\Big) \geq 0 \qquad\forall (q,r)\in \left( \Sing_f (T) \cap \overline{\bB}_{\bar \eta} \right) \times ]0,\bar r[\, .
\end{equation}
\end{proposition}

In particular, we conclude that 
\[
\mathbf{I} (q,0):=\lim_{r\downarrow 0} \mathbf{I} (q,r) 
\]
is a well defined number at every $q\in \mathbf{S}_f:=\Sing_f (T) \cap \overline{\bB}_{\bar\eta} \subset\mathcal{M}$, due to \eqref{e:fine_cm1}. In the fourth step we use Proposition \ref{p:almost-monot} to analyze tangent functions to $N$ at every $q\in \mathbf{S}_f$. Tangent functions are defined following a path analogous to \cite[Section 28]{DLHMS}. We start denoting by $\mathbf{e} (q,\cdot): T_q \mathcal{M}\cap \mathbf{B}_1 \to \mathcal{M}$ the exponential function centered at $q\in \mathbf{S}_f$ and we then introduce the rescaled functions
\[
N_{q,r} (x) := \frac{N (\mathbf{e} (q, rx))}{\left( r^{2-m} \mathbf{D} (q,r)\right)^{\sfrac{1}{2}}}\, , \qquad x \in T_q\mathcal{M} \cap \bB_1\,. 
\]
We then use the blow-up argument of \cite[Section 28]{DLHMS}, and we combine the conclusion of Proposition \ref{p:decay-improved} with a classification of the possible frequency values of homogeneous Dir-minimizing special $Q$-valued functions (see Theorem \ref{thm:linear-what-is-needed} below) to achieve the following

\begin{proposition}[Blow-up]\label{p:blow-up}
Let $T$ be as in Proposition \ref{p:almost-monot}. For every $q\in \mathbf{S}_f$ and any sequence $r_k\downarrow 0$ there is a subsequence, not relabeled, such that
\begin{itemize}
    \item[(i)] $N_{q,r_k}$ converges, strongly in $W^{1,2}_{loc}$, to a Dir-minimizing special $Q$-valued function $\bar{N} \colon B_1 \subset \R^m \to \mathscr{A}_Q(\R)$, with nonzero Dirichlet energy;
    \item[(ii)] $\mathbf{\eta}\circ \bar N = 0$ and $\bar N (0) = Q \a{0}$; 
    \item[(iii)] $\bar N$ is $\mathbf{I} (q,0)$-homogeneous;
    \item[(iv)] $\mathbf{I} (q,0)$ is an integer strictly larger than $1$.
\end{itemize}
\end{proposition}

We next stratify $\mathbf{S}_f$ using the values of $\mathbf{I} (q,0)$. More precisely, for every integer $k\geq 2$ we introduce 
\[
\mathbf{S}_f (k) := \{q\in \mathbf{S}_f: \mathbf{I} (q,0) = k\}\, .
\]
The following useful fact follows from the upper semicontinuity of the map $\mathbf{I}(\cdot,0)$.

\begin{lemma}\label{l:closed}
Let $T$ be as in Proposition \ref{p:almost-monot}. Then for every $q\in \mathbf{S}_f (k)$ there is a closed ball $\overline{\mathbf{B}}_r (q)$ with the property that $\mathbf{S}_f (k) \cap \overline{\mathbf{B}}_r (q)$ is a closed set.
\end{lemma}

In order to prove Theorem \ref{t:even-reduced-2} it will suffice to show that, for every integer $k \geq 2$,
$\mathcal{H}^{m-2+\delta} (\mathbf{S}_f (k)) =0$ for every $\delta > 0$, and $\mathbf{S}_f (k)$ is discrete when $m=2$. As in the standard Federer's reduction argument we will argue by contradiction and assume the latter is false. Standard measure theory would then lead to the following:
\begin{itemize}
    \item[(a)] either $m\geq 3$ and there are $\delta > 0$ and, for some $k \ge 2$, a point $\bar{q}\in \mathbf{S}_f (k)$ with the property that
\begin{equation}\label{e:lower-density}
\limsup_{r\downarrow 0} \frac{\mathcal{H}^{m-2+\delta}_\infty (\mathbf{S}_f (k) \cap \mathbf{B}_r (\bar q))}{r^{m-2+\delta}} > 0\, ;
\end{equation}
    \item[(b)] or $m=2$ and $\bar q\in \bS_f (k)$ is not isolated.
\end{itemize}

The convergence in Proposition \ref{p:blow-up} will then be used to infer, from the validity of the above contradiction assumption (a)-(b), the following conclusion.

\begin{proposition}\label{p:reduction_argument}
Assume $T$ is as in Proposition \ref{p:blow-up} and assume that (a) or (b) as above hold. Then, there are a sequence $r_j\downarrow 0$ and a tangent function $\bar N$ as in Proposition \ref{p:blow-up} with the additional property that there is an $(m-1)$-dimensional subspace $V \subset \R^m$ such that $\bar N$ is invariant under translations along any $v\in V$.
\end{proposition}

However, a Dir-minimizing special $Q$-valued function $\bar N$ as in Proposition \ref{p:reduction_argument} does not exist, due to the validity of the following structural theorem on the singular set of $\Dir$-minimizing special $Q$-valued functions with target $\R$. We refer the reader to \cite{DLHMS_linear} and to Section \ref{s:linear_thms} for the relevant notation concerning special multiple-valued functions.

\begin{theorem} \label{thm:linear-what-is-needed}
Let $\Omega \subset \R^m$ be open and connected, and let $u \in W^{1,2}(\Omega, \Iqspecl)$ be $\Dir$-minimizing with $\Dir(u,\Omega) > 0$ and $\etab \circ u \equiv 0$. Then the following holds.
\begin{itemize}
    \item[(i)] The singular set $\sing (u)$ can be decomposed as the disjoint union $\sing_{m-1} (u) \cup \Sigma (u)$, where $\Sigma (u)$ is a closed set with Hausdorff dimension at most $m-2$ and
$\sing_{m-1} (u)$ is an $(m-1)$-dimensional submanifold of $\Omega$ of class $C^{1,\alpha}$, where $\alpha=\alpha(m,Q)$; moreover $\Sigma (u)$ is a discrete set if $m=2$.
\item[(ii)] For every $x \in \sing (u)$, the value $I_{x,u} (0)$ of the frequency function is a positive integer, and $\sing_{m-1}(u) = \left\lbrace x \in \sing (u) \, \colon \, I_{x,u}(0) = 1 \right\rbrace$.
\item[(iii)] In particular, if $u$ is $\alpha$-homogeneous, then $\alpha = I_{x,u}(0)$ is a positive integer.
\end{itemize}
\end{theorem}

Combined with Proposition \ref{p:reduction_argument}, this in particular shows that (a) or (b) lead to a contradiction: there is no $\Dir$-minimizing $\alpha$-homogeneous special $Q$-valued function with target $\R$ admitting an $(m-1)$-dimensional subspace of singular points of frequency $\alpha \ge 2$. Hence, this completes our argument.

\medskip

The proof of Theorem \ref{thm:linear-what-is-needed} will occupy the first part of the paper, from Section \ref{s:linear_thms} to Section \ref{s:linear_proofs}. In fact, we will provide a very precise description of the asymptotic behavior of a special $\Dir$-minimizing map $u$ at points of the top-dimensional part $\sing_{m-1}(u)$ of the singular set, see Theorem \ref{thm:main}. Notice that, as claimed in (ii), we are able to identify $\sing_{m-1}(u)$ ``analytically'' via the frequency function. In addition, as specified in (iii), we are able to classify all possible homogeneous $\Dir$-minimizing functions, which in turn provide the tangent functions at singular points: in particular, we have some nontrivial information also on the points which belong to $\Sigma(u) = \sing(u) \setminus \sing_{m-1}(u)$. The key behind such classification is that homogeneous $\Dir$-minimizers can be related in a clean way to homogeneous (classical) harmonic polynomials, see Theorem \ref{thm:main3} below, hence the conclusion that the possible degrees of homogeneity are integers.    

\medskip

The second part of the paper will instead contain the proofs of all the other facts stated in this synopsis. In particular: Proposition \ref{p:decay-improved} is proved in Section \ref{s:improved-decay}; Section \ref{s:supercm} contains the proof of Proposition \ref{p:contact}; Propositions \ref{p:almost-monot} and \ref{p:blow-up} are proved in Section \ref{s:monot}; finally, in Section \ref{s:conclusion} we prove Lemma \ref{l:closed} and Proposition \ref{p:reduction_argument}, and then we complete the arguments for Theorem \ref{t:even} and Theorem \ref{t:even-structure}.

\newpage

\part{Linear theory}

\section{Overview} \label{s:linear_thms}

The goal of this part is to prove Theorem \ref{thm:linear-what-is-needed} on $\Dir$-minimizing \emph{special} $Q$-valued functions with target $\R$. The reader can refer to \cite{DLHMS_linear} for the relevant terminology and notation. In particular, the set of special $Q$-points in $\R^n$ is denoted $\mathscr{A}_Q(\R^n)$, and it is defined as the quotient
\[
\mathscr{A}_Q(\R^n) := \mathcal{A}_Q(\R^n) \sqcup \mathcal{A}_Q(\R^n) / \sim
\]
of the disjoint union $\mathcal{A}_Q(\R^n) \sqcup \mathcal{A}_Q(\R^n) := \{(T,\varepsilon)\,\colon\,T \in \mathcal{A}_Q(\R^n),\, \varepsilon \in \{-1,1\}\}$ (here, $\mathcal{A}_Q(\R^n)$ denotes the set of classical $Q$-points in $\R^n$, see \cite{DLS_Qvfr}) through the equivalence relation
\begin{eqnarray*}
(T,1) \sim (S,1) &\iff& T=S\,, \\
(T,-1) \sim (S,-1) &\iff& T=S\,, \\
(T,1) \sim (S,-1) &\iff& \mbox{$T=Q\a{P}=S$ for some $P \in \R^n$}\,.
\end{eqnarray*}
One endows $\mathscr{A}_Q(\R^n)$ with the metric space structure generated by the pseudometric in $\mathcal{A}_Q(\R^n) \sqcup \mathcal{A}_Q(\R^n)$ given by
\[
\G_s((T,\alpha),(S,\beta))^2:=
\begin{cases}
\G(T,S)^2 & \mbox{if $\alpha=\beta$}\,,\\
\abs{T\ominus\bfeta(T)}^2 + \abs{S\ominus\bfeta(S)}^2 + Q \abs{\bfeta(T)-\bfeta(S)}^2 & \mbox{otherwise}\,,
\end{cases}
\]
where, for $T=\sum \a{P_i}$ and $S=\sum\a{\tilde P_i}$
\[
\G (T,S)^2 := \min\left\lbrace \sum_{i=1}^Q \abs{P_i - \tilde P_{\sigma(i)}}^2 \, \colon \, \mbox{$\sigma$ is a permutation of $\{1,\ldots,Q\}$} \right\rbrace\,,
\]
and where $\bfeta(T):=Q^{-1}\sum_i P_i \in \R^n$, $\abs{T}:= \G(T, Q\a{0})$, and, for $a\in\R^n$, $T\ominus a := \sum_i \a{P_i-a}$.

Let $\Omega \subset \R^m$ be open and bounded. A (Borel) measurable map $u \colon \Omega \subset \R^m \to \mathscr{A}_Q(\R^n)$ induces a pair of maps $u^\pm \colon \Omega \to \mathcal{A}_Q(\R^n)$ defined by \begin{equation} \label{d:pm}
u^\pm (x) := 
\begin{cases}
T & \mbox{if $u(x)=(T,\pm 1)$}\,,\\
Q\a{\bfeta(T)} & \mbox{if $u(x)=(T,\mp 1)$}\,.
\end{cases}
\end{equation}
Furthermore, we define $\bfeta \circ u \colon \Omega \to \R^n$ by setting $\bfeta \circ u(x)=\bfeta(T)$ for $u(x)=(T,\varepsilon)$, and we say that $u \in W^{1,2}(\Omega,\mathscr{A}_Q(\R^n))$ if and only if $\bfeta \circ u \in W^{1,2}(\Omega,\R^n)$ and $u^\pm \ominus \bfeta \circ u \in W^{1,2}(\Omega,\mathcal{A}_Q(\R^n))$. The Dirichlet energy is then the functional on the Sobolev space $W^{1,2}(\Omega, \mathscr{A}_Q(\R^n))$ defined by
\[
\Dir(u,\Omega) := \Dir (u^+ \ominus \bfeta \circ u, \Omega) + \Dir (u^- \ominus \bfeta \circ u, \Omega) + Q \Dir (\bfeta \circ u,\Omega)
\]

Next, we denote by $\{\Omega^+,\Omega^-,\Omega_0\}$ the canonical decomposition of $\Omega$ induced by $u$, that is
\begin{equation} \label{d:canonical}
\Omega^{\pm} := \left\lbrace x\in\Omega \, \colon \, \abs{u^{\pm}(x) \ominus \bfeta\circ u(x)} > 0 \right\rbrace\,, \qquad \Omega_0 := \Omega \setminus \left( \Omega^+ \cup \Omega^- \right)\,.
\end{equation}
Given a map $u \colon \Omega \to \mathscr{A}_Q(\R^n)$ minimizing the Dirichlet energy on $\Omega \subset \R^m$, a point $x \in \Omega$ is regular for $u$ if it is either $u(y)=(u^+(y),1)$ or $u(y)=(u^-(y),-1)$ for all $y$ in a neighborhood of $x$ and if, moreover, $x$ is regular for the corresponding classical $\Dir$-minimizing map $u^+$ or $u^-$. When the target dimension is $n=1$, \cite{DLHMS_linear} shows that, if $\Omega$ is connected, then either $\Omega$ coincides with $\Om_0$ (in which case $u$ coincides with $Q$ copies of a single classical harmonic function), or $\Om_0$ has Hausdorff dimension at most $m-1$ and coincides with the singular set of $u$ (namely, the complement in $\Omega$ of the set of regular points), which from now on we will denote by $\sing (u)$. In fact it can be seen that, under the latter alternative, $\Om_0$ must necessarily have dimension \emph{equal} to $m-1$, otherwise it is empty and the whole domain $\Omega$ coincides either with $\Om^+$ or with $\Om^-$. Here we achieve a complete understanding of the asymptotic behavior of $u$ at the $(m-1)$-dimensional stratum of $\sing (u)$. More precisely, we will prove the following theorem.

\begin{theorem}\label{thm:main}
Let $\Omega \subset \mathbb R^m$ be an open domain and $u\in W^{1,2} (\Omega, \Iqspecl)$ a $\Dir$-minimizing function. Then the following holds.
\begin{itemize}
\item[(a)] $u$ is locally Lipschitz, and in fact we have the quantitative estimate
\begin{equation}\label{e:Lip-est}
\|Du\|_{L^\infty (B_r (x))} \leq C r^{-\sfrac{m}{2}} \|Du\|_{L^2 (B_{2r} (x))} \qquad \forall B_{2r} (x) \subset \Omega\, ,
\end{equation}
where $C$ is a dimensional constant.
\item[(b)] $\sing (u)$ can be decomposed as $\sing_{m-1} (u) \cup \Sigma (u)$, where $\Sigma (u)$ is a closed set with Hausdorff dimension at most $m-2$ and
$\sing_{m-1} (u)$ is an $(m-1)$-dimensional submanifold of $\Omega$ of class $C^{1,\alpha}$, where $\alpha=\alpha(m,Q)$; moreover $\Sigma (u)$ is a discrete set if $m=2$.
\item[(c)] Let $x\in \sing_{m-1} (u)$ and $B_r (x)$ be so that $\sing_{m-1} (u)\cap B_r (x) = \sing (u) \cap B_r (x)$ and $B_r (x)\setminus \sing_{m-1} (u)$ consists of two connected components $E_+$ and $E_-$ diffeomorphic to a ball. Then
\begin{itemize}
\item[(c1)] $u|_{E_\pm} = \left(\sum_{i=1}^Q \a{u^\pm_i}, \pm 1\right)$ where $u^\pm_1\leq u^\pm_2 \leq \ldots \leq u^\pm_Q$ are harmonic functions which are of class $C^{1,\alpha}$ up to $\sing_{m-1} (u) \cap B_r (x)$;
\item[(c2)] If $u^\pm_j (y) = u^\pm_k (y)$ for some $y\in E_\pm$, then $u^\pm_j \equiv u^\pm_k$ on $E_\pm$;
\item[(c3)] On $\sing (u) \cap B_r(x)$ the following conditions hold:
\begin{eqnarray}
u^+_1 = \ldots = u^+_Q&=&u^-_1 = \ldots = u^-_Q\,,  \label{e:Dirichlet_condition} \\
\sum_{i=1}^Q |\nabla u^+_i|^2 &=& \sum_{i=1}^Q |\nabla u^-_i|^2\label{e:transmission_condition}\, .
\end{eqnarray} 
\end{itemize}
\end{itemize}
\end{theorem}

\begin{remark}
A standard application of the result in \cite{DeSa} implies that $\sing_{m-1}(u)$ is actually a $C^\infty$ submanifold.
\end{remark}

\begin{remark}[Connection to Free-Boundary problems] \label{rmk FB} 
Conclusion (c3) of Theorem \ref{thm:main} suggests a possible interpretation of our problem as a Free-Boundary problem: the Free-Boundary is given by the set $\sing_{m-1}(u)$, where the harmonic sheets of $u$ satisfy both the Dirichlet condition \eqref{e:Dirichlet_condition} and the transmission condition \eqref{e:transmission_condition}. Indeed, two closely related problems that motivated our study are segregation problems as studied by Caffarelli-Lin (see \cite{CaLi}) and the so called Double-Phase problem as studied by Alt-Caffarelli-Friedman (see \cite{AlCaFr}) and De Silva-Ferrari-Salsa (see \cite{DeFeSa}). It is interesting to notice, however, that in the above works the proof of the analogous result as Theorem \ref{thm:main} follows a viscosity approach (see \cite{CaLi, DeFeSa, DeSpVe}), while our proof is variational and based on an epiperimetric inequality approach. Similar ideas have been used in the treatment of the Double-Phase problem in dimension two \cite{SpVe}. However, the technique does not extend to higher dimensions due to the presence of cuspidal/branching points; the main result of the present paper is, instead, valid in any dimension, since such points are not present.
\end{remark}

\subsection{Strategy of the proof}

From \cite{DLHMS_linear} it is already known that $u$ is H\"older continuous. Without loss of generality, we make the following

\begin{ipotesi}\label{ass:main}
$\Omega$ is connected, $\Dir (u,\Omega) >0$, and $u$ has zero average everywhere, namely $\bfeta\circ u \equiv 0$. 
\end{ipotesi} 

\begin{remark}\label{r:sheeting} As anticipated, from the theory developed in \cite{DLHMS_linear} we know that $\sing (u)$ has Hausdorff dimension at most $m-1$ and, under Assumption \ref{ass:main}, it can be described as
\begin{equation} 
\sing (u) = \{u = (Q \a{0} ,1) = (Q \a{0}, -1)\}\, 
\end{equation}
(and since in $\Iqspecl$ the points $(Q \a{P}, 1)$ and $(Q\a{P}, -1)$ are identified, from now on we will use simply the notation $Q\a{P}$). Indeed if $u(x) \neq Q \a{0}$, we know by continuity (cf. \cite[Section 8]{DLHMS_linear}) that there is a neighborhood $U$ of $x$ which is either contained in $\Omega^+$ or in $\Omega^-$. In particular $u|_U$ is given by $(w, 1)$ or $(w, -1)$ for some $\Dir$-minimizing function $w$ taking values in $\mathcal{A}_Q (\R)$; see \cite{DLS_Qvfr}. The latter has no singularities and consists of $\sum_{i=1}^Q \a{w_i}$ for classical harmonic functions $w_i$ satisfying:
\begin{itemize}
\item $w_1\leq w_2\leq \ldots \leq w_Q$;
\item either $w_i<w_{i+1}$ or $w_i\equiv w_{i+1}$. 
\end{itemize}
\end{remark}

Both in the theory developed in \cite{DLHMS_linear} and towards a proof of Theorem \ref{thm:main}, a pivotal role is played by the \emph{frequency function} of $u$. For any $x \in \Omega$, the frequency function of $u$ at $x$ is defined by
\[
r \in \left( 0, \dist (x, \partial \Omega) \right) \mapsto I_{x,u}(r) := \frac{r D_{x,u}(r)}{H_{x,u}(r)}\,,
\]
where
\[
D_{x,u}(r) := \Dir(u,B_r(x)) \quad \mbox{and} \quad H_{x,u}(r) := \int_{\partial B_r (x)} \G(u(y),Q\a{0})^2 \, d\Ha^{m-1}(y)\,.
\]

For the sake of simplicity, we will drop the subscript $_u$ whenever this does not give rise to ambiguity. The crucial properties of the frequency functions, proved in \cite[Theorem 9.2]{DLHMS_linear} and valid for any target $\R^n$, can be summarized as follows:
\begin{itemize}
\item[(i)] $r \mapsto I_{x}(r)$ is non-decreasing, so that, in particular, the limit $\lim_{r\to 0^+} I_{x}(r) =: I_{x}(0)$ exists and it is finite;
\item[(ii)] Setting, for brevity, $I_x=I_{x}(0)$, one has that $I_x=0$ if and only if 
\[
\max\{|u^+(x)|, |u^-(x)|\}>0\, ;
\]
\item[(iii)] If $u(x) = Q \a{0}$ (and thus, by (ii), $I_x \ne 0$), then $I_x \ge c_0$ for some constant $c_0$ depending only on $m,n$, and $Q$;
\item[(iv)] $I_{x}(\rho) \equiv I_x$ for $\rho \in \left(0,r\right)$ if and only if $\left.u\right|_{B_r(x)}$ is positively homogeneous of degree $I_x$. 
\end{itemize}

Theorem \ref{thm:main} will follow from the following two results, where we achieve a classification of the possible values of $I_{x,u}(0)$ when the target of $u$ is $\R$.

\begin{theorem}\label{thm:main2}
Assume $u$ is as in Theorem \ref{thm:main} and satisfies Assumption \ref{ass:main}. Then, for every $x\in \sing (u)$ the value of the frequency function $I_x (0)$ is a positive integer, and $\sing_{m-1} (u) = \{x\in \sing (u): I_x (0)=1\}$. 
\end{theorem}

\begin{theorem}\label{thm:main3}
If $u\in W^{1,2}_{{\rm loc}} (\R^m, \Iqspecl)$ is locally $\Dir$-minimizing, satisfies Assumption \ref{ass:main} and is $\alpha$-homogeneous, then
\begin{itemize}
\item[(a)] $\alpha= I_0 (0)$ is a positive integer;
\item[(b)] denoting by $u^+$ and $u^-$ the positive and negative part of $u$ (see \eqref{d:pm}), $|u^+|-|u^-|$ is a harmonic polynomial $p$ of degree $\alpha$, and $\sing (u)$ is its zero set;
\item[(c)] if we denote by $\Om^\pm_j$ the connected components of $\{\pm p > 0\}$, then there are vectors $A^\pm_j = (A^\pm_{j,1}, \ldots , A^\pm_{j,Q})\in \mathbb R^Q$ such that
\begin{equation}
u (x) = \left\{
\begin{array}{ll}
\left(\sum_i \a{A^\pm_{j,i} \, p (x)}, \pm 1\right) \qquad &\mbox{for every $x\in \Om^\pm_j$}\\ \\
Q\a{0} \qquad &\mbox{for every $x\in \{p=0\}$}\,.
\end{array}\right.
 \end{equation}  
\end{itemize}
\end{theorem}

\begin{remark}
Regarding the conclusions of the previous theorem, observe that $\bfeta\circ u \equiv 0$ is equivalent to $\sum_i A^\pm_{j,i}=0$ for all $j$, and that the transmission condition \eqref{e:transmission_condition} is equivalent to:
\begin{itemize}
\item[(T)] if $\overline{\Om_j^+} \cap \overline{\Om_k^-} \cap \sing_{m-1} (u) \neq \emptyset$, then $|A^+_j|=|A^-_k|$.
\end{itemize} 
\end{remark}

In order to pass from Theorem \ref{thm:main2} and Theorem \ref{thm:main3} to the main result Theorem \ref{thm:main}, the key step is a uniqueness of blow up with decay, see Proposition \ref{p:decay}, whose proof is achieved through an epiperimetric inequality approach, see Lemma \ref{l:epiperimetric}. Notice that the statement of Theorem \ref{thm:linear-what-is-needed} is contained in the statements of Theorems \ref{thm:main}, \ref{thm:main2}, and \ref{thm:main3}.

\section{Inductive procedure and base step}\label{ss:base}

We will prove the three theorems by induction over the dimension $m$. 

\subsection{Base step $m=1$} First of all, we recall the following elementary facts (see \cite{DLHMS_linear}): 
\begin{itemize}
\item[(i)] $\sing (u)$ consists of isolated points for any $\Dir$-minimizer $u$;
\item[(ii)] if $v$ is an homogeneous $\Dir$-minimizer with $\bfeta\circ v \equiv 0$, then there are $2Q$ real constants $a_1, \ldots , a_Q , b_1 , \ldots , b_Q$ such that 
\[
v (x) = \left\{\begin{array}{ll}
\left(\sum_i \a{a_i x}, 1\right) \qquad &\mbox{for $x\geq 0$}\\ \\
\left(\sum_i \a{b_i x}, -1\right) &\mbox{for $x\leq 0$}
\end{array}\right.\, ,
\]
up to a choice of orientation of the real line.
\end{itemize}
It follows immediately from $\bfeta\circ v = 0$ that $\sum_i a_i = \sum_i b_i =0$. Note that, if we assume that $u$ is not identically equal to $Q\a{0}$, then $\{|u|\neq 0\}$ consists precisely of the two connected components $\Om^+ := (0, \infty)$ and $\Om^- = (-\infty, 0)$. So, if we define the vectors $A^+_1 := (a_1, \ldots , a_Q)$ and $A^-_1 := (b_1, \ldots, b_Q)$, in order to achieve the conclusion of Theorem \ref{thm:main3} we just need to show that $|A^+_1|^2 = |A^-_1|^2$. Observe that this follows immediately from the inner variation \cite[Proposition 7.1]{DLHMS_linear}, which in this specific case simplifies to
\[
\int |Dv|^2\, \varphi' \, dx = 0 \qquad \mbox{for every $\varphi \in C^1_c(\R)$}\,,
\]
thus showing that $|Dv|^2$ must be constant, and so Theorem \ref{thm:main3} holds for $m=1$. 

Now, for a general (not necessarily homogeneous) minimizer $u$ with $\bfeta \circ u =0$ and positive Dirichlet energy, consider, without loss of generality, the case in which its domain is an open interval $I$. Recalling the definition of the sets $\Om^+$ and $\Om^-$, by the discussion above we conclude that:
\begin{itemize}
\item[(a)] $\Om^+$ and $\Om^-$ are both the union of a collection of intervals $I_j^+$ and $I_j^-$.
\item[(b)] If an endpoint $\gamma$ of $I_j^+$ (respectively $I_j^-$) is not an endpoint of $I$, then it must be common to some $I_k^-$ (resp. $I_k^+$).
\item[(c)] On each interval $I_j^+$ (resp. $I_j^-$) $u$ takes respectively the form $(\sum_i \a{a_{j,i} x + c_{j,i}}, 1)$ (resp. $(\sum_i \a{b_{j,i} x + d_{j,i}}, -1)$) for constants $a_{j,i}, c_{j,i}$ (resp. $b_{j,i}, d_{j,i}$) with the properties 
\[
\sum_i a_{j,i} = \sum_i c_{j,i} = 0 \qquad \qquad \left(\mbox{resp.}\quad \sum_i b_{j,i} = \sum_i d_{j,i} = 0 \right)\, \qquad \forall\,j .
\] 
Any two of the linear functions cannot cross in $I_j^+$ (resp. $I_j^-$), whereas they all must vanish at an endpoint of the interval if the latter is in the interior of $I$. We then conclude that either $\sing (u) = \emptyset$ (and $I = \Om^+$, or $I = \Om^-$) or  $\sing (u)$ consists of precisely one point. In the latter case the map $u$ is just the translation of a $1$-homogeneous $\Dir$-minimizer. In particular this shows that all the conclusions of Theorem \ref{thm:main} and Theorem \ref{thm:main2} hold. 
\end{itemize}

\subsection{Inductive statement} In the remaining sections of this first part we therefore focus on proving the following inductive statement.

\begin{proposition}\label{p:inductive}
Assume Theorem \ref{thm:main3} hold for $m$. Then Theorems \ref{thm:main}, \ref{thm:main2} hold for $m$ and Theorem \ref{thm:main3} hold for $m+1$.  
\end{proposition}

\section{Lipschitz regularity of Dir-minimizers}\label{s:Lipschitz}

In this section we show that the Lipschitz regularity is a consequence of the classification of the possible values for the frequency function. First of all, observe that we can reduce the proof of \eqref{e:Lip-est} to the case where, additionally, $\etab\circ u =0$. Indeed, observe that in general $\etab \circ u$ satisfies the inequality
\[
\int Q\, |D(\etab \circ u)|^2 \leq \int |Du|^2 
\]
and it is a harmonic function, from which it readily follows that
\[
\|D(\etab\circ u)\|_{L^\infty (B_r(x))} \leq C r^{-m/2} \|D(\etab\circ u)\|_{L^2 (B_{2r}(x))}\, .
\]

\begin{corollary}\label{c:Lipschitz}
Assume $u\in W^{1,2} (\Omega, \Iqspecl)$ satisfies Assumption \ref{ass:main} and $I_{x} (0) \geq 1$ for every $x\in \sing (u)$. Then $u$ is locally Lipschitz. 
\end{corollary}

\begin{proof}Through a classical Morrey-type argument (cf. \cite[Section 8]{DLHMS_linear}), it suffices to show the estimate
\begin{equation}\label{e:decay1}
\int_{B_{r} (x)} |Du|^2 \leq C(m) \frac{r^{m}}{\rho^{m}} \int_{B_\rho (y)} |Du|^2 \qquad \forall x\in B_{\rho/8} (y)\, , \forall r< \frac{\rho}{8}\, .
\end{equation}
In turn, we notice that \eqref{e:decay1} is implied by
\begin{equation}\label{e:decay2}
\int_{B_\sigma (z)} |Du|^2 \leq \frac{\sigma^{m}}{s^{m}} \int_{B_s (z)} |Du|^2 \qquad \forall \sigma <s \quad \mbox{if $u (z) = Q \a{0}$.}
\end{equation}
Indeed, assume \eqref{e:decay2} and fix $x\in B_{\rho/8} (y)$ and $r<\frac{\rho}{8}$. If $u (x) = Q \a{0}$, then it suffices to apply \eqref{e:decay2} with $s=\frac{\rho}{4}$ and $\sigma = r$. If $u$ is free of singularities in $B_{\rho} (y)$, then the desired estimate follows from estimates for the classical harmonic functions by Remark \ref{r:sheeting}. Otherwise, let $z$ be the closest point to $x$ with the property that $u (z)=Q\a{0}$ and set $\tau:= \abs{x-z}$. If $\tau\geq \frac{\rho}{8}$, again from Remark \ref{r:sheeting} and classical regularity for harmonic functions we conclude
\[
\int_{B_r (x)} |Du|^2 \leq C \frac{r^{m}}{\rho^{m}} \int_{B_{\rho/8} (x)} |Du|^2
\] 
 and thus \eqref{e:decay1}. 
Consider next the case when $\tau < \frac{\rho}{8}$. If $r\geq \tau$, we can use directly \eqref{e:decay2} to estimate
\[
\int_{B_r (x)} |Du|^2 \leq C \int_{B_{2r} (z)} |Du|^2 \leq C \frac{r^{m}}{\rho^{m}} \int_{B_{\rho/2} (z)} |Du|^2\, ,
\]
while when $r\leq \tau$ we use the classical theory of harmonic functions and then \eqref{e:decay2}:
\begin{equation*}
\int_{B_r (x)} |Du|^2 \leq  C \frac{r^{m}}{\tau^{m}} \int_{B_\tau (x)} |Du|^2
\leq C \frac{r^{m}}{\tau^{m}} \int_{B_{2\tau} (z)} |Du|^2\\
\leq C \frac{r^{m}}{\tau^{m}} \frac{\tau^{m}}{\rho^{m}} \int_{B_{\rho/2} (z)} |Du|^2\, .
\end{equation*}
We finally come to \eqref{e:decay2}. Without loss of generality assume $z=0$ and $s=1$. Set 
\[
h (\rho) := \frac{1}{\rho^{m-1}} \int_{\partial B_\rho} |u|^2 
\]
and, recalling \cite[Proposition 9.3]{DLHMS_linear}, compute
\begin{equation}\label{e:h'}
h' (\rho) = \frac{2}{\rho^{m-1}} \int_{B_\rho} |Du|^2\,.
\end{equation}
In particular
\[
\rho \frac{h'(\rho)}{h (\rho)} = 2 \frac{\rho \int_{B_\rho} |Du|^2}{\int_{\partial B_\rho} |u|^2} = 2I_0(\rho)\, . 
\]
Next, using the monotonicity of $I_0(\rho)$ and the assumption $I_0\geq 1$, we derive
\[
\frac{d}{d\rho} \log h (\rho) \geq \frac{2 I_0 (0)}{\rho} \geq \frac{2}{\rho}\, .
\] 
Integrating the latter inequality on the interval $\left(\sigma,1\right)$ we have
\[
\frac{h(1)}{h (\sigma)} \geq \frac{1}{\sigma^2}\, ,
\] 
i.e.
\[
\int_{\partial B_\sigma} |u|^2 \leq \sigma^{m+1} \int_{\partial B_1} |u|^2 \qquad \forall\,\sigma < 1\,.
\]
Using again the monotonicity of the frequency function, we infer
\[
\int_{B_\sigma} |Du|^2 \leq \frac{1}{\sigma} \int_{B_1} |Du|^2 \frac{\int_{\partial B_\sigma} |u|^2}{\int_{\partial B_1} |u|^2} \leq \sigma^m \int_{B_1} |Du|^2\, .
\] 
\end{proof}

\section{Weiss' functional and its decay when $I_x (0) =1$}

In this section we introduce Weiss' functional, prove its monotonicity and show that it converges to $0$ with a power rate at points $x$ where $I_x(0)$ equals $1$.

\begin{definition}\label{d:Weiss}
Given a $\Dir$-minimizing function $u$ on $\Omega\subset \mathbb R^m$ and a ball $B_r (x)\subset \Omega$, we set $I:= I_x (0)$ and we define the Weiss functional as
\begin{equation}\label{e:Weiss}
\begin{split}
W_{x,u} (r) :&= r^{-(m+2I-2)} D_{x,u} (r) - I r^{-(m+2I-1)} H_{x,u} (r)\\
&=\frac{1}{r^{m+ 2I-2}} \int_{B_r (x)} |Du|^2 - \frac{I}{r^{m+2I-1}} \int_{\partial B_r (x)} |u|^2.
\end{split}
\end{equation}
\end{definition}

Notice that $W_{x,u}(r) \ge 0$ as a consequence of the monotonicity of the frequency function. If $u$ is clear from the context, we will simply write $W_x (r)$. The main goal of this section is the proof of the following monotonicity and decay result for the Weiss functional.

\begin{proposition}\label{p:Weiss}
If $u$ is $\Dir$-minimizing in $\Omega\subset \mathbb R^m$ with $m \ge 2$ and $B_r (x)\subset \Omega$, then $r\mapsto W_x (r)$ is absolutely continuous and
\begin{equation}\label{e:Weiss_monot}
\frac{d}{dr} W_x (r) \geq 0\, .
\end{equation}
Moreover, for every positive $C>0$ there are geometric constants $\varepsilon (m, C), \alpha (m)>0$ such that the following holds. If $I = I_x (0)=1$ and, for some $r>0$ with $B_r(x) \subset \Omega$, $W_x (r) \leq \varepsilon r^{-m-1} H_x(r)$ and $D_x(2r) \leq Cr^{-1} H_x(r)$, then
\begin{equation}\label{e:Weiss_decay}
W_x (s) \leq \frac{s^\alpha}{r^\alpha} W_x (r) \qquad \mbox{for every $0<s <r 
$}\,.
\end{equation}
\end{proposition}

We will divide the proof of Proposition \ref{p:Weiss} in four steps: the monotonicity of $W_x(r)$, its decay at points $x$ of frequency $I_x(0)=1$ assuming epiperimetric inequality (see Lemma \ref{l:epiperimetric}), the classification of $1$-homogeneous blow-ups, and, finally, the proof of the epiperimetric inequality.

\subsection{Monotonicity and proof of \eqref{e:Weiss_monot}} Without loss of generality we assume $x=0$ and drop the subscript $x$ from the $W,D$, and $H$ functionals.  By a standard scaling argument, it suffices to show that $W' (1) \ge 0$. We then compute
\[
W' (1) = - (m+2I-2)\, D (1) + D'(1) + I \,(m+2I-1) \,H (1) - I\, H' (1)\, .
\]
We next use \cite[Eq. (9.5)]{DLHMS_linear} and \cite[Eq. (9.4)]{DLHMS_linear} together with the obvious identity $D'(1) = \int_{\partial B_1} |Du|^2$ to compute
\begin{align}
W' (1) &= - (m+4I-2)\, D(1) + D'(1) + 2\,I^2\, H(1)\nonumber\\
& = - 2 (m+2I-2)\, D(1) + 2\int_{\partial B_1} |Du|^2 - 2\int_{\partial B_1} |\partial_\nu u|^2 + 2I^2\int_{\partial B_1} |u|^2\nonumber\\
&= -2 (m+2I-2) D(1) + 2 \int_{\partial B_1} \left(|D_\tau u|^2 + I^2 |u|^2\right)\, , \label{e:W'-partial}
\end{align}
where $\partial_\nu u$ and $D_\tau u$ denote the normal and tangential derivatives of $u$ at $\partial B_1$. 
We now introduce the notation $u^I$ for the $I$-homogeneous extension of $u$ from $\partial B_1$ to $B_1$, namely:
\begin{equation} \label{d:hom_ext}
u^I (x) := \sum_{i=1}^Q \left(\llbracket r^I u_i (r^{-1} x)\rrbracket, \epsilon (r^{-1} x)\right)\,, \qquad \mbox{where $r:=\abs{x}$ and $\varepsilon(\cdot) \in \{-1,1\}$}\,.
\end{equation}
A straightforward computation yields
\[
\int_{B_1} |Du^I|^2 = \frac{1}{m + 2I -2}\int_{\partial B_1} \left(|D_\tau u|^2 + I^2 |u|^2\right)\, .
\]
Inserting the latter identity in \eqref{e:W'-partial} then gives
\begin{equation}\label{e:W'}
W' (1) = 2 (m+2I-2) \int_{B_1} \left(|Du^I|^2 - |Du|^2\right)\, .
\end{equation}
The minimality of $u$ implies therefore $W' (1) \geq 0$. \qed 

\subsection{Epiperimetric inequality and proof of the decay \eqref{e:Weiss_decay}} The key to the decay property \eqref{e:Weiss_decay} is given by the following lemma, whose proof is inspired by \cite{SpVe}.

\begin{lemma}[Epiperimeteric Inequality]\label{l:epiperimetric} There is a positive constant $\delta$, depending only on the dimension $m$, with the following property. 
For every fixed positive $M$ there is a positive $\gamma = \gamma (M, m) < 1$ such that the following holds.
Assume $u\in W^{1,2} (B_2, \Iqspecl)$ is $\Dir$-minimizing, $I:=I_0 (0)=1$, $W (1) \leq \gamma$, $H(1) =1$ and $D (2) \leq M$. Then, for $u^I$ as in \eqref{d:hom_ext} one has
\begin{equation}\label{e:epiperimetric}
\int_{B_1} \left(|Du^I|^2 - |Du|^2\right) \geq \delta\, W (1)\, . 
\end{equation} 
\end{lemma}

Next, we show how to conclude the proof of Proposition \ref{p:Weiss} from Lemma \ref{l:epiperimetric}.

\begin{proof}[Proof of Proposition \ref{p:Weiss}]
As usual, and without loss of generality, we assume $x=0$. We also suppose that $I=I_0(0)=1$. Let $C$ be any given positive constant and fix a larger $M$, whose choice will be specified later. Let then $\gamma=\gamma(M,m)$ be given by Lemma \ref{l:epiperimetric}. Combining \eqref{e:W'} and \eqref{e:epiperimetric} and using the scaling invariance  of the problem we conclude that, if $W (s) \leq \gamma s^{-m-1} H(s)$ and $D(2s) \leq M s^{-1} H(s)$, then
\begin{equation}\label{e:differential}
\frac{W' (s)}{W(s)} \geq \frac{2m \delta}{s}\, .
\end{equation}
Let $\varepsilon < \gamma$ to be fixed, and choose $r \in \left(0,1\right)$ as in the second part of Proposition \ref{p:Weiss}, in particular such that $W(r) \leq \varepsilon r^{-m-1} H(r)$ and $D(2r) \leq M\,r^{-1}\, H (r)$. Define
\[
s_0:= \inf \{\sigma\leq r  : W (s) \leq \gamma s^{-m-1} H(s) \;\mbox{and}\; D(2s) \leq M s^{-1} H(s)\;\; \forall s\in (s_0, r)\}\,.
\]
Clearly $s_0 < r$ by continuity of $W$, $D$ and $H$. Moreover, for all $s\in (s_0, r]$ we have that \eqref{e:differential} holds; hence, integrating the inequality, and setting $\alpha:= 2m\delta$, we conclude
\begin{equation}\label{e:from_above}
W (s) \leq \frac{s^\alpha}{r^\alpha} W(r) \leq \frac{\varepsilon s^\alpha}{r^{m+1+\alpha}} H (r) \qquad \forall s\in (s_0, r]\, .
\end{equation}
Note that the first part of \eqref{e:from_above} would imply the desired decay \eqref{e:Weiss_decay} if we could show that $s_0=0$. To this aim, we observe that, should $s_0$ be positive, it would either be $W (s_0) = \gamma s_0^{-m-1} H (s_0)$ or $D (2s_0) = M s_0^{-1} H (s_0)$.

Now observe, again using \cite[Eq. (9.5)]{DLHMS_linear} together with the monotonicity of the Weiss functional, that
\begin{equation}\label{e:der_H}
\frac{d}{ds} \left(\frac{H(s)}{s^{m+1}}\right) = \frac{2}{s} W (s) \leq \frac{2 \varepsilon}{s^{1-\alpha}} \frac{H(r)}{r^{m+1+\alpha}}\, . 
\end{equation}
Hence, integrating for $s$ between $\sigma$ and $r$ we conclude
\begin{equation}\label{e:from_below}
\frac{H (\sigma)}{\sigma^{m+1}}\geq \frac{H(r)}{r^{m+1}} \left(1 - \frac{2\varepsilon}{\alpha} + \frac{2\varepsilon \sigma^\alpha}{\alpha r^\alpha}\right)
\geq  \frac{H(r)}{r^{m+1}} \left(1 - \frac{2\varepsilon}{\alpha}\right)\, \qquad \forall \sigma \in (s_0, r)\, .
\end{equation}
In particular, combining \eqref{e:from_above} and \eqref{e:from_below} we conclude
\[
W(\sigma) \leq \frac{\varepsilon}{1-\frac{2\varepsilon}{\alpha}} \frac{H (\sigma)}{\sigma^{m+1}} \qquad \forall \sigma \in (s_0, r]\, .
\]
Hence, if we choose $\varepsilon$ small enough compared to $\gamma$ and $\alpha$, and so dimensional, we infer 
\begin{equation}\label{e:threshold_1}
W (\sigma) \leq \frac{\gamma}{2} \, \frac{H(\sigma)}{\sigma^{m+1}}\qquad  \forall \sigma \in (s_0, r].
\end{equation}
In particular, if $s_0$ were positive it should be $D (2s_0) = M s_0^{-1} H (s_0)$, since \eqref{e:threshold_1} implies that the alternative $W(s_0)=\gamma s_0^{-m-1}H(s_0)$ cannot occur.

Next, for $\sigma \in (\max \{s_0, \frac{r}{2}\}, r]$ \eqref{e:from_below} implies
\[
H (\sigma) \geq 2^{-m-1} H (r) \left(1-\frac{2\varepsilon}{\alpha}\right) \geq 2^{-m-2} H (r)
\] 
(again assuming $\varepsilon$ is chosen appropriately small). On the other hand, since $D(2\sigma) \leq D(2r) \leq C r^{-1} H (r)\leq C \sigma^{-1} H (r)$. 
In particular
\[
D (2\sigma) \leq 2^{m+2} \sigma^{-1} C H (\sigma)\,,
\]
and thus by choosing $M \geq 2^{m+3} C$ we conclude $s_0 < \frac{r}{2}$. Now, for $\sigma \in (s_0, \frac{r}{2}]$, we have from \eqref{e:from_below} that
\[
H (\sigma) \geq 2^{-m-1} \left(1-\frac{2\varepsilon}{\alpha}\right) H (2\sigma) \geq 2^{-m-2} H (2\sigma)\, .
\]
On the other hand, by the definition of the Weiss functional
\[
D (2\sigma) = (2\sigma)^{-1} H (2\sigma) + (2\sigma)^m W (2\sigma) \leq (2\sigma)^{-1} (1+ \gamma) H (2\sigma) \leq 2^{m+1} (1+\gamma) \sigma^{-1} H (\sigma) \, .
\]
Hence, since $\gamma < 1$, it suffices to impose $M \geq 2^{m+2}$ to conclude that $s_0$ cannot be positive.

Summarizing, $s_0=0$ and \eqref{e:from_above} holds for all $s\in (0,r]$, namely we have proved \eqref{e:Weiss_decay}.
\end{proof}

\subsection{Classification of $1$-homogeneous $\Dir$-minimizers} Before coming to the proof of Lemma \ref{l:epiperimetric} we first observe that

\begin{lemma}\label{l:homogeneous}
If $u\in W^{1,2}_{{\rm loc}} (B_1, \Iqspecl)$ is a $1$-homogeneous locally $\Dir$-minimizing function with $\etab\circ u = 0$ then, up to a suitable change of coordinates,
\begin{equation}\label{e:1-homogeneous}
u(x)=\left\{
\begin{array}{ll}
\left(\sum_i \a{a_i x_1}, +1\right) \qquad &\mbox{if $x_1 \geq 0$}\\ \\
\left(\sum_i \a{b_i x_1}, -1\right) &\mbox{of $x_1\leq 0$}
\end{array}
\right.
\end{equation} 
with $a_1 \leq \ldots \leq a_Q$, $b_1 \leq \ldots \leq b_Q$, $\sum_i a_i =\sum_i b_i = 0$ and $|a|=|b|$. 
\end{lemma}

\begin{proof}
We can use the homogeneity to extend the map to $\Omega=\mathbb R^m$ and observe that it must be locally $\Dir$-minimizing. Consider any connected subset $E\subset \Om^+$. On it we must have $u (x) = \left(\sum_i \a{v_i \cdot x}, +1\right)$ for some vectors $v_1, \ldots , v_Q$. Moreover the functions $L_i (x) = v_i \cdot x$ can be ordered as $L_1 \leq \ldots \leq L_Q$ and either $L_j < L_{j+1}$ or $L_j \equiv L_{j+1}$. If it were $L_1 = L_Q$, then we would have an open set where $u \equiv Q \a{0}$, thereby implying that $u$ must be trivial. Hence $L_1<L_Q$. Consider now a \emph{maximal} connected open subset $E$ of $\Om^+$. Due to the maximality of $E$, we necessarily must have $\partial E\subset \sing (u)$. Thus, if we denote by $H$ the hyperplane $\{x\cdot (v_1-v_Q) =0\}$, we must have $\partial E \subset H$. In particular , since $E$ is connected, either it is empty, or it is one of two connected components of $\mathbb R^m \setminus H$. If $E$ were empty, than $\Om^+$ would necessarily be empty, implying again that $u$ is trivial. So, up to changing coordinates, we can assume that $E = \{x_1>0\}$. Since the same argument applies for any maximal connected open subset of $\Omega^+$ and $\Omega^-$, and since, for the same reason given above, $\Omega^-$ cannot be empty, the only possibility is then that $\Omega^+ = \{x_1 > 0\}$ and $\Omega^-  = \{x_1 < 0 \}$. It is then obvious that \eqref{e:1-homogeneous} above holds. Next, $\sum_i a_i = \sum_i b_i =0$ is an obvious consequence of $\etab\circ u = 0$, while we can use the inner variations, namely \cite[Theorem 7.1]{DLHMS_linear} to conclude that $|Du|^2$ must be constant, which implies $|a|^2 = |b|^2$.   
\end{proof}

\subsection{Proof of Lemma \ref{l:epiperimetric}} 
By the minimality of $u$, it suffices to show the existence of an extension $w$ of $u|_{\partial B_1}$ to $B_1$ such that
\begin{equation}\label{e:comparison1}
\int_{B_1} |Du^I|^2 - \int_{B_1} |Dw|^2 \geq \delta \left(\int_{B_1} |Dw|^2 - \int_{\partial B_1} |u|^2\right)\, .
\end{equation}
But in fact we will show a stronger statement, namely the existence of a $w$ such that
\begin{equation}\label{e:comparison2}
\int_{B_1} |Du^I|^2 - \int_{B_1} |Dw|^2 \geq \delta \left(\int_{B_1} |Du^I|^2 - \int_{\partial B_1} |u|^2\right)\, .
\end{equation}
In order to show the existence of $w$, we will first use the fact that $u$ is close to a $1$-homogeneous minimizer to partition $\partial B_1$ in a suitable collection of open subsets.

\medskip 

Recall that the frequency function is upper semicontinuous in the following sense: if $u_k$ is a sequence of $\Dir$-minimizing maps on the same domain $\Omega$ converging to some $u_\infty$ in $L^2$, $x_k \to x$, $\etab \circ u_k =0$ and $u_k (x_k) = Q\a{0}$, then $\bfeta \circ u_\infty = 0$, $u_\infty (x) = Q\a{0}$, and 
\[
I_{x,u_\infty} (0) \geq \limsup_{k\to \infty} I_{x_k,u_k} (0)\, .
\]
Such upper semicontinuity follows from the monotonicity property of the frequency and the compactness of $\Dir$-minimizers in the strong local topology of $W^{1,2}$.

Therefore, if $u_k\in W^{1,2} (B_2, \Iqspecl)$ is a sequence with $\etab\circ u_k = 0$, $u_k (0) = Q\a{0}$, $I_{0,u_k} (0) =1$, $W_{0,u_k} (1)\downarrow 0$, $H_{0,u_k} (1) =1$, and $D_{0,u_k} (2)\leq M$ then $u_k$ converges, up to subsequences, to a $\Dir$-minimizer $u_\infty$ with $\etab\circ u_\infty =0$, $u (0) = Q\a{0}$, $I_{0,u_\infty}(0)\geq 1$, and
\[
\int_{B_1} |Du_\infty|^2 \leq \int_{\partial B_1} |u_\infty|^2\, ,
\]
which in turn by the monotonicity of the frequency function implies that $I_{0,u_\infty} (r) \equiv 1$ for all $r \in \left[0,1\right]$. In particular we conclude from \cite[Theorem 9.2]{DLHMS_linear} that $u_\infty$ is $1$-homogeneous and, up to a change of coordinates, it takes the form described in Lemma \ref{l:homogeneous}. Since $H_{0,u_\infty} (1) = 1$, we infer that $|Du_\infty| = c_m = |B_1|^{-1/2}$. In particular, using the notation of \eqref{d:pm},
\[
(|u_\infty^+| - |u_\infty^-|) (x) = c_m x_1\, .
\]
Observe also that the convergence is in $C^{0, \beta} (B_{3/2})$ for some $\beta (m)>0$, by \cite[Theorem 8.1]{DLHMS_linear}. Hence, if we let $G_k^+$ and  $G_k^-$ be the largest (with respect to Hausdorff measure) connected components of the sets $\{|u_k^+|> 0\}\cap \partial B_1$ and $\{|u_k^-|>0\}\cap \partial B_1$, we conclude that, for every $\eta >0$,  there is $k_0=k_0(\eta)$ large enough such that for all $k \ge k_0$:
\begin{align*}
&\partial B_1 \cap \{x_1 > \eta\} \subset G^+_k \subset \partial B_1 \cap \{x_1 > -\eta\}\,,\\
&\partial B_1 \cap \{x_1<-\eta\} \subset G^-_k \subset \partial B_1 \cap \{x_1 <\eta\}\, .
\end{align*} 

Coming back to the map $u$ of the Lemma, and denoting $\Omega=B_2$, set $F^\pm:= \Om^\pm \cap \partial B_1$, and order their connected components $F^\pm_j$ so that $\mathcal{H}^{m-1} (F^\pm_j) \geq \mathcal{H}^{m-1} (F^\pm_{j+1})$. From the discussion above it follows that, for every fixed $\eta>0$, provided $\gamma$ is chosen small enough, we have
\begin{align*}
&\partial B_1 \cap \{x_1 > \eta\} \subset F^+_1 \subset \partial B_1 \cap \{x_1 > -\eta\}\\
&\partial B_1 \cap \{x_1<-\eta\} \subset F^-_1 \subset \partial B_1 \cap \{x_1 <\eta\}\\
&\bigcup_{j\geq 2}  (F^-_j\cup F^+_j) \subset \partial B_1 \cap \{-\eta \le x_1 \le \eta\}\, .
\end{align*}
In particular, for any fixed $\kappa>0$, if we denote by $\lambda_k (\Gamma)$ the eigenvalues of the Dirichlet Laplacian on a domain $\Gamma \subset \partial B_1$ (as usual ordered 
so that $\lambda_{k+1} (\Gamma) \geq \lambda_k (\Gamma)$), we have, provided $\gamma$ is chosen small enough, that
\begin{eqnarray}
|\lambda_1 (F_1^\pm)-(m-1)|&\leq& \kappa,\label{e:eigen_1}\\
\lambda_k (F_1^\pm) &\geq& 2m - \kappa \qquad \forall k \geq 2,\label{e:eigen_2}\\
\lambda_k (F_j^\pm) &\geq& \kappa^{-1} \qquad\quad\;\; \forall k\geq 1, \forall j\geq 2\, .\label{e:eigen_3}
\end{eqnarray}
In the latter we have used the fact that the first and second eigenvalues of the Dirichlet Laplacian on the half-sphere $\partial B_1\cap \{x_1>0\}$ are $m-1$ and $2m$ and the monotonicity of each eigenvalue of the Dirichlet Laplacian with respect to set inclusion. We next wish to use the inequalities \eqref{e:eigen_1}, \eqref{e:eigen_2} and \eqref{e:eigen_3} to construct a suitable competitor $w$ for $u$ and this will be done in three steps. In order to describe them we fix the notation $U_j^\pm$ for the cones 
\[
\{t\, x: x\in F_j^\pm, t \in ]0,1[\}
\]
and we let $\sigma$ be a positive parameter smaller than $\frac{1}{2}$ whose choice will be specified later. 
\begin{itemize}
\item[(a)] In the first step we define $w$ on the open sets $U^\pm_j\cap B_1\setminus \overline{B_\sigma}$ for $j\geq 2$. An essential feature is that $w$ will vanish identically on $U^\pm_j \cap \partial B_\sigma$.
\item[(b)] In the second step we define $w$ on the open sets $U^+_1\cap B_1\setminus \overline{B_\sigma}$ and $U^-_1\cap B_1\setminus \overline{B_\sigma}$. If we denote by $z^\pm$ the first eigenfunction of the Dirichlet Laplacian on $F^\pm_1$, an essential feature of the extension is the existence of two vectors $\xi^\pm \in \mathbb R^Q$ such that
\begin{align}
w (\sigma x) &= \left(\sum_{i=1}^Q \a{\xi_i^\pm z^\pm (x)}, \pm 1\right) \qquad \forall x\in F_1^\pm\, .
\end{align}
\item[(c)] In the third step we extend $w$ to $B_\sigma$ using the special form that it has on $\partial B_\sigma$. 
\end{itemize}
In each step we compute the energy of the extension $w$ and compare it in a suitable way to the energy of the $1$-homogeneous extension $u^I$. We will then conclude the desired inequality in the final, fourth, step.

\medskip

{\bf Step 1.} Fix $F:= F_j^+$ and $U:= U_j^+$ with $j\geq 2$, the construction for $F_j^-$ being entirely analogous. Let $u_1\leq \ldots \leq u_Q$ be such that
\[
\left.u\right|_F = \left(\sum_i \a{u_i}, 1\right)\, .
\]    
We let $w_i$ be the harmonic function on $V := U\cap B_1\setminus \overline{B_\sigma}$ which equals $u_i$ on $F$ and vanishes identically on the remaining portion of the boundary. We use the spectrum of the Dirichlet Laplacian on $F$ to compute $w_i$ explicitely in polar coordinates $(r, \theta)$. Thus we let $\lambda_k=\lambda_k(F)$ be the corresponding eigenvalues, $h_k=h_k(F)$ be a choice of the corresponding eigenfunctions which gives an orthonormal base in $L^2$ and
\begin{equation}\label{e:Fourier}
\left.u_i\right|_F (\theta) = \sum_k a_{i,k} \, h_k (\theta)
\end{equation}
be the Fourier expansion of $u_i$, with $a_{i,k}=a_{i,k}(F)$. The function $w_i$ is given, in polar coordinates, by
\[
w_i (r, \theta) = \sum_k a_{i,k} \,  \rho_k (r) \, h_k (\theta) = \sum_k a_{i,k} (A_k r^{\mu_k} - B_k r^{-\mu_k}) h_k (\theta) 
\]
where 
\begin{equation}\label{e:exp=eigen}
\mu_k (m-2 +\mu_k) = \lambda_k
\end{equation} 
and the coefficients $A_k$ and $B_k$ are decided by the boundary conditions $\rho_k (1) =1$ and $\rho_k (\sigma) =0$ and thus
\begin{equation}\label{e:computation_Ak_Bk}
A_k = \frac{\sigma^{-\mu_k}}{\sigma^{-\mu_k}-\sigma^{\mu_k}}\,,\qquad B_k = \frac{\sigma^{\mu_k}}{\sigma^{-\mu_k} - \sigma^{\mu_k}}\, .
\end{equation}
Moreover
\begin{align}
\int_F |u_i|^2 &= \sum_k a_{i,k}^2\,, \qquad 
\int_F |D_\tau u_i|^2  = \sum_k \lambda_k a_{i,k}^2\, .
\end{align}
We next compute the Dirichlet energy of $w_i$. Using the harmonicity and the fact that it vanishes on $\partial V\setminus F$, we have $\int_V |Dw_i|^2 =
\int_F w_i \frac{\partial w_i}{\partial \nu}$ which, using the orthogonality of the eigenfunctions, becomes
\[
\int_V |Dw_i|^2 = \sum_k \rho_k (1)\, \rho_k'(1)\,  a_{i,k}^2 = \sum_k \mu_k\, \frac{\sigma^{\mu_k} + \sigma^{-\mu_k}}{\sigma^{-\mu_k} - \sigma^{\mu_k}}\, a_{i,k}^2\, .
\]
Finally, we wish to compute the energy of the $1$-homogeneous extension $u^I_i$ of $u_i|_F$. Such energy is given by
\begin{equation}\label{e:energy_1_hom}
\int_{U} |Du^I_i|^2 = \frac{1}{m} \int_F (|D_\tau u_i|^2 + |u_i|^2) = \sum_k \frac{1}{m}\, (\lambda_k +1)\, a_{i,k}^2\, , 
\end{equation}
and 
\begin{equation} \label{e:energy_1_hom_outside}
    \int_{V} |Du^I_i|^2 = (1-\sigma^m)\, \int_{U} |Du^I_i|^2 = (1-\sigma^m)\, \sum_k \frac{1}{m}\, (\lambda_k +1)\, a_{i,k}^2\, .
\end{equation}
Hence, for every $\delta\in \left(0,1\right)$ we have
\[
\begin{split}
 &\int_V |Du^I_i|^2 - \int_V |Dw_i|^2 - \delta \int_U |Du^I_i|^2 + \delta \int_F |u_i|^2\\
 &\qquad\qquad = \sum_k \left[ \frac{(\lambda_k+1) (1-\sigma^m-\delta) +m\delta}{m} - \mu_k\, \frac{1+ \sigma^{2\mu_k}}{1-\sigma^{2\mu_k}}\right] a_{i,k}^2\, .
\end{split}
\]
Now, recalling that $\lambda_k \geq \kappa^{-1}$ by \eqref{e:eigen_3}, we have from \eqref{e:exp=eigen} that $\frac12 \lambda_k \leq \mu_k^2 \leq \lambda_k$ provided $\kappa$ is chosen sufficiently small. Hence, if $\delta,\sigma \leq \frac{1}{2(m+1)}$ we can estimate
\[
\frac{(\lambda_k +1)(1-\sigma^m-\delta)+m\delta}{m} - \mu_k \, \frac{1+ \sigma^{2\mu_k}}{1-\sigma^{2\mu_k}}
\geq \frac{\lambda_k}{(m+1)} - 3\,\sqrt{\lambda_k}\, .
\]
Clearly, the latter expression is positive if $\kappa$ is small enough depending on $m$. We thus conclude the existence of a $\kappa_0=\kappa_0(m) >0$ such that
\[
\frac{(\lambda_k +1)(1-\sigma^m-\delta)+m\delta}{m} - \mu_k \, \frac{1+ \sigma^{2\mu_k}}{1-\sigma^{2\mu_k}}\geq 0 
\qquad
\mbox{if $\delta\leq  \frac{1}{2(m+1)}$, $\sigma \leq \frac{1}{2(m+1)}$ and $\kappa \leq \kappa_0$.}
\]
In particular, summing over $i$ and recalling that $F = F_j^+$ (or, analogously, $F=F_j^-$) we conclude
that, for $j\geq 2,\, \delta \leq \frac{1}{2(m+1)},\, \sigma \leq \frac{1}{2(m+1)}$, and $\kappa \leq \kappa_0$,
\begin{equation}\label{e:small_domains_inequality}
\int_{U_j^\pm\cap B_1 \setminus \overline{B_\sigma}} |Du^I|^2 - \int_{U_j^\pm \cap B_1 \setminus \overline{B_\sigma}} |Dw|^2 \geq \delta \left(\int_{U_j^\pm} |Du^I|^2 - \int_{F_j^\pm} |u|^2\right)\, .
\end{equation}

\medskip

{\bf Step 2.} We now fix $F = F_1^+$ (or analogously $F= F_1^-$) and again define the domains $U$ and $V$, the functions $u_1 \leq \ldots \leq u_Q$ and denote by:
\begin{itemize}
\item $\lambda_k=\lambda_k(F)$ the eigenvalues of the Dirichlet Laplacian;
\item $\mu_k=\mu_k(F)$ the positive numbers defined through the relation \eqref{e:exp=eigen};
\item $h_k=h_k(F)$ an orthonormal base of eigenfunctions;
\item $a_{i,k}=a_{i,k}(F)$ the Fourier coefficients of one fixed $u_i$, as in \eqref{e:Fourier}.
\end{itemize}
Again we wish to extend $\left.u_i\right|_F$ to $w_i$ on $V$. This time the extension is done so that:
\begin{equation}\label{e:boundaries}
w_i (x) = u_i (x) \quad \mbox{and} \quad w_i (\sigma x) = \sigma a_{i,1} h_1 (x) \quad \forall x\in F
\end{equation} 
and
\begin{equation}\label{e:boundaries2}
\left.w_i\right|_{\partial V \setminus (F\cup \sigma F)} \equiv 0\, .
\end{equation}
The explicit formula is then given, in polar coordinates, by 
\[
w_i (r, \theta) = \underbrace{a_{i,1}\, r \, h_1 (\theta)}_{:=\psi} + \underbrace{\sum_{k\geq 2} a_{i,k} (A_k r^{\mu_k} - B_k r^{\mu_k}) h_k (\theta)}_{:=\varphi}\, , 
\]
where the coefficients $A_k$ and $B_k$ are given again by the formulas \eqref{e:computation_Ak_Bk}. Note that $\varphi$ is harmonic and vanishes on $\partial V \setminus F$, while $\psi$ is $1$-homogeneous. Using orthogonality of the eigenfunctions $h_k$, it is immediate to check that $\int_V |Dw_i|^2 = \int |D\psi|^2 + \int |D\varphi|^2$. We can thus use the computations of the previous step to conclude that
\[
\int_{V} |Dw_i|^2 = \frac{(\lambda_1+1) (1-\sigma^m)}{m} a_{i,1}^2 + \sum_{k\geq 2} \mu_k \,\frac{1+ \sigma^{2\mu_k}}{1-\sigma^{2\mu_k}}\, a_{i,k}^2\, .
\]
As for the $1$-homogeneous extension $u_i^I$ of $u_i$ to $U$, the formula \eqref{e:energy_1_hom} remains valid. 

We now conclude the step by summing over $i$ and writing the result in each of the two different regions $F_1^\pm$. For this reason we introduce the notation $a_{i,k}^\pm$, $h_k^\pm$, $\lambda_k^\pm$ and $\mu_k^\pm$, which is selfexplanatory, as well as $a_k^\pm = \left(a^\pm_{1,k}, \ldots , a^\pm_{Q,k} \right) \in \R^Q$. We then deduce the identity
\begin{align}
\int_{U^\pm_1\cap B_1 \setminus \overline{B_\sigma}} |Dw|^2 &= \frac{(\lambda^\pm_1+1) (1-\sigma^m)}{m} \, \abs{a^\pm_1}^2 + \sum_{k\geq 2} \mu^\pm_k\, \frac{1+ \sigma^{2\mu^\pm_k}}{1-\sigma^{2\mu^\pm_k}}\, \abs{a_k^\pm}^2\label{e:corona_inequality}
\end{align}

\medskip

{\bf Step 3} We now wish to extend $w$ to $B_\sigma$. Observe that 
\[
w (\sigma x) = \left\{
\begin{array}{ll}
\left(\sum_i \a{\sigma a^\pm_{i,1} h_1^\pm (x)}, \pm 1 \right) \qquad &\mbox{if $x\in F_1^\pm$}\\
Q \a{0} \qquad &\mbox{ if $x\not\in F_1^+\cup F_1^-$.}
\end{array}\right.
\]
Using the notation $a^\pm_1 \in \R^Q$ introduced above, consider on $\partial B_\sigma$ the function
\[
\tilde{u} (\sigma x) := \sigma \, \abs{a^+_1} \, h_1^+ (x)\, \mathbf{1}_{F_1^+} (x) - \sigma\, \abs{a^-_1}\, h_1^- (x)\, \mathbf{1}_{F_1^-} (x)\, ,
\]
where we assume that $h_1^\pm$ are positive on $F_1^\pm$: recall that, being $h_1^\pm$ an eigenfunction for the first eigenvalue of the Dirichlet Laplacian on $F_1^\pm$, it does not vanish in $F_1^\pm$. Next, let $\mathfrak h$ be the harmonic extension of $\tilde u$ to $B_\sigma$. Observe that $\{\mathfrak h>0\}\cap B_\sigma$ must necessarily consist of one connected component. Indeed, $\sigma\,F_1^+$ is clearly contained in the boundary of a single connected component of $\{\mathfrak h > 0\} \cap B_\sigma$; should there be a distinct connected component $E$, $\mathfrak h$ would be then identically equal to zero on $\partial E$, thus violating the maximum principle. Likewise, $\{\mathfrak h<0\} \cap B_\sigma$ consists of a single connected component, too. It is thus immediate to verify that the following is a well-defined map in $W^{1,2} (B_\sigma, \Iqspecl)$ which extends $w|_{\partial B_\sigma}$:
\[
w (x) =
\left\{
\begin{array}{ll}
\left(\sum_i \a{{\textstyle{\frac{a_{i,1}^+}{|a_1^+|}}} \mathfrak h (x)}, +1\right) \qquad & \mbox{if $\mathfrak h (x) >0$}\\ \\
\left(\sum_i \a{-{\textstyle{\frac{a_{i,1}^-}{|a_1^-|}}} \mathfrak h (x)}, -1\right) \qquad &\mbox{if $\mathfrak h (x) <0$}\\ \\
Q \a{0} \qquad &\mbox{otherwise}\,.
\end{array}\right.
\] 
Observe that 
\begin{equation} \label{e:sono uguali}
\int_{B_\sigma} |Dw|^2 = \int_{B_\sigma} |D\mathfrak h|^2\, ,
\end{equation}
Now, the Weiss' epiperimetric inequality for classical harmonic functions (which the reader can easily prove expanding $\tilde{u}$ in spherical harmonics and using their harmonic extensions) implies the existence of a positive $\delta_*$ such that
\begin{align}
\int_{B_\sigma} |D\tilde{u}^I|^2 - \int_{B_\sigma} |D\mathfrak h|^2 &\geq \delta \left(\int_{B_\sigma} |D \tilde{u}^I|^2 - \int_{\partial B_\sigma} |\tilde{u}|^2\right)
\qquad \forall \delta < \delta_*\, .
\end{align}
On the other hand we can explicitly compute
\[
\int_{B_\sigma} |D\tilde{u}^I|^2  - \int_{\partial B_\sigma} |\tilde{u}^I|^2 = \sigma^m \left(\abs{a_1^+}^2\left(\frac{\lambda^+_1+1}{m} -1 \right) + \abs{a_1^-}^2\left(\frac{\lambda^-_1+1}{m}-1\right)\right)
\]
Thus we achieve
\begin{equation}\label{e:inner_inequality}
\int_{B_\sigma} |Du^I|^2 - \int_{B_\sigma} |Dw|^2 \geq \delta\,  \sigma^m \left(\abs{a_1^+}^2\left(\frac{\lambda^+_1+1}{m} -1 \right) + \abs{a_1^-}^2\left(\frac{\lambda^-_1+1}{m}-1\right)\right)
\qquad \forall \delta\leq \delta_*\, ,
\end{equation}

where we have used that $\int_{B_\sigma} \abs{Du^I}^2 \ge \int_{B_\sigma} \abs{D \tilde u^I}^2$ together with \eqref{e:sono uguali}.

\medskip

{\bf Step 4} Combining \eqref{e:inner_inequality} with \eqref{e:corona_inequality} and \eqref{e:small_domains_inequality}, we conclude that, if $\sigma\leq \frac{1}{2(m+1)}$, $\delta_1, \delta_2 \leq \delta_*$ and $\kappa\leq \kappa_0$, then
\begin{align}
\int_{B_1} |Du^I|^2 - \int_{B_1} |Dw|^2
\geq & \delta_1 \sum_{j\geq 2} \left(\int_{U_j^+\cup U_j^-} |Du^I|^2 - \int_{F_j^+\cup F_j^-} |u|^2\right)\nonumber\\
& + \delta_2 \,  \sigma^m \left(\abs{a_1^+}^2\left(\frac{\lambda^+_1+1}{m} -1 \right) + \abs{a_1^-}^2\left(\frac{\lambda^-_1+1}{m}-1\right)\right)\nonumber\\
&+ \sum_{k\geq 2} \sum_{\epsilon =+,-}   |a_k^\epsilon|^2\left(\frac{(\lambda^\epsilon_k+1)(1-\sigma^m)}{m} -\mu^\epsilon_k \frac{1+ \sigma^{2\mu^\epsilon_k}}{1-\sigma^{2\mu^\epsilon_k}}\right)\, ,
\end{align}
where, we recall, $\lambda_k^\varepsilon,\,\mu_k^\varepsilon$, and $a_k^\varepsilon$ are relative to the region $F_1^\varepsilon$. Now, using the formula
\[
\mu_k^\pm = \frac{2-m + \sqrt{(m-2)^2 +4\, \lambda_k^\pm}}{2}
\]
it can be readily checked that there are $\bar \kappa, \bar \sigma$ and $\bar \delta$ positive such that
\[
\mu_k^\pm \frac{1+ \sigma^{2\mu^\pm_k}}{1-\sigma^{2\mu^\pm_k}} \leq  \frac{(\lambda^\pm_k+1)(1-\sigma^m-\delta)+m\delta}{m} \qquad \forall k\geq 2
\]
as soon as 
\begin{equation}\label{e:conditions_2}
\delta<\bar\delta, \sigma < \bar\sigma\quad \mbox{and} \quad \kappa < \bar\kappa\, .
\end{equation}
Since the latter implies,
\[
\frac{(\lambda^\pm_k+1)(1-\sigma^m)}{m} -\mu^\pm_k \frac{1+ \sigma^{2\mu^\pm_k}}{1-\sigma^{2\mu^\pm_k}}
\geq \delta \left(\frac{\lambda^\pm_k+1}{m}-1\right)\, ,
\]
we conclude the existence of a $\sigma^\star$, $\kappa^\star$ and $\delta^\star$ such that, for $\delta_1, \delta_2, \delta_3<\delta^\star$, $\kappa < \kappa^\star$ and $\sigma<\sigma^\star$,
\begin{align*}
\int_{B_1} |Du^I|^2 - \int_{B_1} |Dw|^2
\geq & \delta_1 \sum_{j\geq 2} \left(\int_{U_j^+\cup U_j^-} |Du^I|^2 - \int_{F_j^+\cup F_j^-} |u|^2\right)\nonumber\\
& + \delta_2 \, \sigma^m \left(|a_1^+|^2\left(\frac{\lambda^+_1+1}{m} -1 \right) + |a_1^-|^2\left(\frac{\lambda^-_1+1}{m}-1\right)\right)\nonumber\\
&+ \delta_3 \sum_{k\geq 2} \left( |a_k^+|^2 \left(\frac{\lambda^+_k+1}{m}-1\right) + |a_k^-|^2 \left(\frac{\lambda^+_k+1}{m}-1\right)\right)\, .
\end{align*}
We thus fix $\sigma< \sigma^\star$ and $\kappa < \kappa^\star$ and set $\delta_3 = \delta_2 \sigma^m = \delta_1 < \delta^\star$. Taking into consideration that
\[
\int_{U_1^\pm} |Du^I|^2 - \int_{F_1^\pm} |u|^2 = \sum_{k\geq 1} |a_k^\pm|^2 \left(\frac{\lambda^\pm_k+1}{m}-1\right)\, ,
\]
we thus conclude the desired inequality
\[
\int_{B_1} |Du^I|^2 - \int_{B_1} |Dw|^2 \geq \delta_1 \left(\int_{B_1} |Du^I|^2  - \int_{\partial B_1} |u|^2\right)\, .
\]

\section{Decay and regularity at points with frequency $1$}

The main point of this section is to use the decay of the Weiss functional (Proposition \ref{p:Weiss}) to conclude the uniqueness of tangent maps at points where the frequency is $1$.

\begin{proposition}\label{p:decay}
Let $u\in W^{1,2} (B_2, \Iqspecl)$ be as in Assumption \ref{ass:main} with $u(0) = Q \a{0}$, $I=I_0 (0) =1$ and $D (2) \leq C H (1) = C$ for some positive real number $C$. Let $\alpha (m)$ and $\varepsilon (m,C)$ be the constants of Proposition \ref{p:Weiss}, and assume in addition that $W_0 (1) \leq
\bar{\varepsilon} \, H (1) \leq \varepsilon \, H(1)$. Then, if $\bar \varepsilon$ is sufficiently small depending only on $m$, there is a unique $1$-homogeneous $\Dir$-minimizing function $u_0\in W^{1,2}_{{\rm loc}} (\mathbb R^m, \Iqspecl)$ and positive constants $\bar C (m,C)$, $\beta (m)$ such that
\begin{equation}\label{e:uniform_decay}
\|\cG_s (u, u_0)\|_{C^0 (B_r)} \leq \bar C r^{1+\beta}\qquad \forall r\in (0,1]\, .
\end{equation}
Moreover, recalling that $L := \sing (u_0)$ is a hyperplane (cf. \cite{DLHMS_linear}), we have
\begin{equation}\label{e:flattening}
\dist_\Ha (\sing (u) \cap B_r, L \cap B_r) \leq \bar C r^{\beta}\, ,
\end{equation}
where $\dist_\Ha$ denotes Hausdorff distance.
\end{proposition}

A simple corollary of the above proposition is the following.

\begin{corollary}\label{c:reg}
There exists $\alpha=\alpha(m)\in\left(0,1\right)$ with the following property. Let $u\in W^{1,2} (\Omega, \Iqspecl)$ be as in Assumption \ref{ass:main} such that $\sing (u) = \{x: I_{x,u} (0) = 1\}$. Then, $\sing (u)$ is a $C^{1,\alpha}$ submanifold of $\Omega$, and at every point $x\in \sing (u)$ there is a unique $1$-homogeneous non-trivial tangent function $u_x$, i.e. $u_x$ is locally $\Dir$-minimizing in $\R^m$ and
\begin{equation}
  \lim_{r\downarrow 0} \frac{1}{r^{m+2}} \int_{B_r (x)} \mathcal{G}_s (u(y), u_x (y-x))^2\, dy = 0\, .  
\end{equation}
\end{corollary}

\subsection{Proof of Proposition \ref{p:decay}} Recalling \eqref{e:der_H} and Proposition \ref{p:Weiss} we have
\[
\frac{d}{ds} \left(\frac{H(s)}{s^{m+1}} \right)= \frac{2\,W(s)}{s} \leq \bar \varepsilon\, s^{\alpha-1} \qquad \mbox{for every $0 < s < 1$}\, .
\]
Integrating the above inequality between $0$ and $1$ and recalling that $H(1) =1$, if $\bar{\varepsilon}$ is sufficiently small (depending on $\alpha$) we conclude the existence of $H_0 \in [\frac{1}{2},1]$ such that
\begin{equation}
H_0 = \lim_{r\downarrow 0} \frac{H(r)}{r^{m+1}}\,.
\end{equation}
Integrating between $0$ and $r$ then yields
\begin{equation} \label{e:decay for H}
0 \leq \frac{H(r)}{r^{m+1}} - H_0 \leq \frac{\bar\varepsilon}{\alpha}\,  r^\alpha \qquad \mbox{for every $0 < r < 1$}\,.
\end{equation} 
Next, we use the identity $2 D(r) = H'(r) - (m-1) \frac{H(r)}{r}$ (see \cite[Proposition 9.3]{DLHMS_linear}) to write
\[
2 \frac{D (r)}{r^m} - 2 H_0 = r\,\frac{d}{dr} \left(\frac{H(r)}{r^{m+1}}\right) + \frac{2H(r)}{r^{m+1}} - 2 H_0\, ,
\]
thus concluding
\begin{equation} \label{e:decay for D}
0 \leq \frac{D(r)}{r^m} - H_0 \leq \frac{2}{\alpha} \bar\varepsilon r^\alpha \qquad \mbox{for every $0 < r < 1$}\,.
\end{equation}
The monotonicity of $r \mapsto I(r)$ together with \eqref{e:decay for H} and \eqref{e:decay for D} then easily give
\[
    0 \le I(r) - 1  = \frac{D(r)}{r^m}\,\frac{r^{m+1}}{H(r)} - 1 \le \left( H_0 + \frac{2}{\alpha} \bar\varepsilon r^\alpha \right) \, \frac{1}{H_0} - 1 \leq C (\alpha)\, \bar\varepsilon\, r^\alpha\, ,
\]
where we have used that $H_0 \geq \frac{1}{2}$. Introduce now the rescaled functions
\[
u_r (x) := \frac{u(rx)}{r}\, .
\]
We claim that, for $0 < s < r$, it holds
\begin{equation} \label{conto monotonia frequency}
\int_{\partial B_1} \cG_s (u_r, u_s)^2 \leq 2\,(r-s) \int_s^r \frac{1}{2t} \frac{d}{dt} \left(\frac{D(t)}{t^m} - H_0\right)\, dt\, .
\end{equation}
Towards the proof of \eqref{conto monotonia frequency}, fix $x\in \partial B_1$, and write
\[
u_s(x) = \left( S, \varepsilon(S) \right)\,, \qquad u_r(x) = \left( R,\varepsilon(R) \right)\,,
\]
with $S,R \in \mathcal{A}_Q(R)$. If $\varepsilon(R)=\varepsilon(S)$, then 
\[
\begin{split}
\cG_s(u_r(x),u_s(x))^2 &= \cG(u_r(x),u_s(x))^2 \leq \sum_{i=1}^Q \left| \int_s^r \frac{d}{dt}\left(\frac{u_i(tx)}{t}\right) \, dt \right|^2 \\ 
&\leq (r-s)\,\sum_{i=1}^Q \int_{s}^r \left| \frac{d}{dt}\left(\frac{u_i(tx)}{t}\right)\right|^2\,dt\,.
\end{split}
\]
If, instead, $\varepsilon(R) = -\varepsilon(S)$ then, by continuity of the map $t \in \left(s,r\right) \mapsto u_t(x) \in \Iqspecl$, there exist $\tau_1,\tau_2 \in \left(s,r\right)$ with $s \leq \tau_1 \leq \tau_2 \leq r$ such that $u_{\tau_k}(x)=Q\a{0}$ for $k=1,2$ and, writing $u_t(x)=\left(T, \varepsilon(T)\right)$, one has 
\begin{align*}
    & \varepsilon(T)=\varepsilon(S) \qquad  \mbox{for every  $t\in \left( s, \tau_1 \right)$}\,, \\
      & \varepsilon(T)=\varepsilon(R) \qquad \mbox{for every  $t\in \left(\tau_2,r \right)$}\,.
\end{align*}
In particular, then
\[
\begin{split}
\cG_s(u_r(x),u_s(x))^2 &\leq 2\,\cG(u_s(x),u_{\tau_1}(x))^2 + 2\,\cG(u_{\tau_2}(x), u_r(x))^2 \\
&\leq 2\, (r-s) \sum_{i=1}^Q \int_s^r \left| \frac{d}{dt} \left( \frac{u_i(tx)}{t} \right) \right|^2\,dt\,,
\end{split}
\]
as above. The estimate \eqref{conto monotonia frequency} is then obtained by following the same computations of \cite[(5.9)]{DLS_Qvfr}, having care of using the variation formulas \cite[(7.7)-(7.8)]{DLHMS_linear} valid in the setting of special $Q$-valued functions on $m$-dimensional domains.

Next, integrating \eqref{conto monotonia frequency} by parts as in \cite[Section 5.2.1]{DLS_Qvfr} we reach the estimate
\[
\int_{\partial B_1} \cG_s (u_r, u_s)^2 \leq C (\alpha, m)\, \bar\varepsilon \, r^\alpha \qquad \mbox{for all $s\leq r$}\, .
\]
In turn, the latter implies that $u_r$ converges in $L^2 (B_1)$ to a unique map $u_0$, and indeed we conclude the decay rate
\[
\|\cG_s (u_r, u_0)\|_{L^2 (B_1)} \leq C (\alpha, m)\, \bar\varepsilon\, r^\alpha\, .
\]
Since 
\[
\int_{B_1} |Du_r|^2 = \frac{D(r)}{r^m}\, ,
\]
the uniform bound on the Dirichlet energy and \cite[Theorem 8.1]{DLHMS_linear} imply the existence of constants $\gamma (m)$, $C (m)$ such that
\[
\|u_r\|_{C^\gamma} \leq C\, .
\]
Consider now the function $\xi (x) := \cG_s (u_r (x), u_0 (x))$ and fix a point $y\in B_{1/2}$. Observe that 
\[
|\xi(y)|\leq C |y-z|^\gamma + |\xi (z)|
\]
and averaging over a ball of radius $\rho$ centered at $y$ we achieve
\[
|\xi (y)|\leq C \rho^\gamma + \frac{C}{\rho^{m/2}} \|\xi \|_{L^2 (B_\rho(y))} \leq C \rho^\gamma + C \bar\varepsilon \frac{r^\alpha}{\rho^{m/2}}\, .
\]
Clearly, choosing $\rho$ to be an appropriate power of $r$, we conclude that \[\|\cG_s(u_r,u_0)\|_{C^0 (B_{1/2})} \leq \bar C r^\beta\] for some $\beta (\alpha, \gamma,m) >0$. In turn this implies \eqref{e:uniform_decay}. 

Next, the $1$-homogeneity of $u_0$ is already proved in \cite{DLHMS_linear} and thus $\sing (u_0) = L$ is a hyperplane by Lemma \ref{l:homogeneous}. Without loss of generality assume that $L = \{x_1=0\}$.  Moreover, consider the coefficients $a_i, b_i$ as in \eqref{e:1-homogeneous}. Since $1\geq H_0 \geq \frac{1}{2}$, 
\begin{equation}\label{e:two-sided-H}
\frac{1}{2} \leq \int_{\partial B_1} |u_0|^2 \leq 1\, .
\end{equation}
Hence a simple compactness argument shows the existence of a geometric constant $\bar{c} (m) >0$ such that
\begin{equation}\label{e:si_stacca}
\min \{\max |b_i|, \max |a_i|\} \geq \bar{c} (m)\, .
\end{equation}
Otherwise there would be a sequence of $1$-homogeneous minimizers for which 
\[
\lim_{i\to \infty}\min \{\max |b_i|, \max |a_i|\} =0
\] 
and \eqref{e:two-sided-H} holds, which up to subsequences would converge to a nontrivial $1$-homogeneous minimizer which vanishes identically on one of the half spaces $\{\pm x_1 \geq 0\}$, which is a contradiction. Now, \eqref{e:si_stacca} implies that
\[
\cG_s (Q \a{0}, u_0 (x)) \geq \bar{c} (m) |x_1|\, .
\]
Since $\sing (u) = \{x: u (x) = Q \a{0}\}$, \eqref{e:uniform_decay} clearly implies that
\begin{equation}\label{e:Hausdorff_one_side}
\sing (u) \cap B_r \subset \{x: |x_1| \leq \bar{C} r^{1+\beta}\}\, ,
\end{equation}
for some constant $\bar{C} (\beta, m)$. On the other hand, using the notation $\Om^+$ and $\Om^-$ of \eqref{d:pm}, we also conclude the existence of a constant $\bar{C}$ such that
\[
(\pm \bar{C} r^{1+\beta}, \bar x) \in \Om^\pm \qquad \mbox{for all $\bar x \in \mathbb R^{m-1}$ with $|\bar x|\leq r$.}
\]
Now, by continuity of $u$, the latter means that, for every $\bar x\in \mathbb R^{m-1}$ with $|\bar x| \leq r$ there is a $\tilde{x} \in [-\bar{C} r^{1+\beta}, \bar{C} r^{1+\beta}]$ such that $u (\tilde x, \bar x) = Q \a{0}$. The latter, combined with \eqref{e:Hausdorff_one_side}, implies \eqref{e:flattening}. \qed

\subsection{Proof of Corollary \ref{c:reg}} Fix $x\in \sing (u)$ with $I_{x}(0)=1$. First assume without loss of generality that $x=0$, and apply the rescaling
\[
v_r (y) := C (r) \,  u (ry)\, ,
\]
where the constant $C (r)$ is chosen so that $H_{v_r,0} (1) =1$. Since, up to subsequences, $v_r$ converges strongly in $W^{1,2}_{{\rm loc}} (B_2)$ to a $1$-homogeneous Dir-minimizer we then conclude that 
\begin{align*}
&\lim_{r\downarrow 0} W_{v_r,0} (1) =0\,,\\
&\lim_{r\downarrow 0} D_{v_r,0} (2) = 2^m\, .
\end{align*}
For a sufficiently small $r$ we can then apply Proposition \ref{p:decay} to conclude that $v_r$ has a unique tangent function at $0$. Obviously this proves the uniqueness of the tangent function to $u$ at $0$.

We wish now to show the existence of a radius $r$ such that $B_r (x) \cap \sing (u)$ is a $C^{1,\beta}$ submanifold. Since $I_{x,u} (0)=1$, 
\[
\lim_{r\downarrow 0} \frac{r D_{x,u} (2r)}{H_{x,u} (r)} = 2^m\, .
\] 
Fix the constant $\bar \varepsilon$ for $C=2^{m+3}$ in Proposition \ref{p:decay} and observe that, since $I_{x,u}(0)=1$, for a sufficiently small $r$ we also have
\[
W_{x,u} (r) \leq \frac{\bar \varepsilon}{2} \, \frac{H_{x,u} (r)}{r^{m+1}}\, .
\]
Fix thus an $r$ such that the latter condition holds and such that $D_{x,u} (2r) \leq 2^{m+1} r^{-1} H_{x,u} (r)$. 

Now, as above consider the rescaled functions
\[
u_r (y) := u (x+ry)\, 
\]
and normalize them to a $v_r$ so to have $H_{v_r,0}(1) =1$. Observe that $\sing (v_r)$ is a $C^{1,\beta}$ submanifold in $B_\rho (0)$ if and only if
$\sing (u) \cap B_{\rho r} (x)$ is a $C^{1,\beta}$ submanifold. Thus, by a slight abuse of notation, keep denoting $v_r$ by $u$.

Observe that we have $D_{0,u} (2) \leq 2^{m+1} H_{0,u}(1) $ and $W_{0,u} (1) \leq \frac{\bar \varepsilon}{2} H_{0,u}(1)$. By continuity, 
\[
D_{x,u} (2) \leq 2^{m+2} H_{x,u} (1) \qquad \mbox{and} \qquad W_{x,u} (1) \leq \bar \varepsilon H_{x,u} (1)
\]
for all $x$ in a neighborhood of $0$. Since, by assumption, every $x\in \sing (u)$ has $I_{x,u} (0) =1$, we can then apply -- modulo re-normalizing the function appropriately as above -- Proposition \ref{p:decay} at each point $x$ in a neighborhood of $0$, and hence we find a unique tangent function at each $x$ and  
an affine hyperplane $L_x$ (which corresponds to the singular set of the tangent function) passing through $x$ such that
\begin{equation}\label{e:flattening1}
\dist_{\Ha} (\sing (u) \cap B_r(x), L_x \cap B_r(x)) \leq \bar C r^{\beta}\, \qquad \forall x\in \sing (u) \cap B_1\, , 
\end{equation}
for constant $\bar C$ and $\beta$ which are independent of the point $x$. 
It is a classical fact that the latter estimate implies the $C^{1,\beta}$ regularity of $\sing (u) \cap B_{1/2}$.

\section{Proof of the Theorems \ref{thm:main}, \ref{thm:main2} and \ref{thm:main3}} \label{s:linear_proofs}

As already mentioned, the proof will be by induction on the dimension $m$. The case $m=1$ has been established in Section \ref{ss:base} and hence we just need to show the inductive step. The key is to prove Theorem \ref{thm:main3}, as explained in the following corollary.

\begin{corollary}\label{c:reg2}
Let $m\geq 1$ be a dimension for which Theorem \ref{thm:main3} holds. Then Theorem \ref{thm:main} and Theorem \ref{thm:main2} hold true in dimension $m$, and additionally they hold true in dimension $m+1$ when $u$ is homogeneous.
\end{corollary}

\begin{proof}[Proof of Corollary \ref{c:reg2}] Observe that $\alpha$ is a possible value for $I_{x,u}(0)$ with $u$ $\Dir$-minimizing on $\Omega \subset \R^m$ if and only if $\alpha$ is a possible degree of homogeneity of a $\Dir$-minimizing map on $\R^m$. For each homogeneous $\Dir$-minimizer $v$, we define its \emph{building dimension} as the maximal number of linear independent directions $\xi \in \R^m$ with the property that $u (x+\xi) = u (x)$ for every $x$. Applying the arguments of \cite[Section 3.6.2]{DLS_Qvfr},  we find that if $u$ is a $\Dir$-minimizer on a domain $\Omega \subset \mathbb R^{m}$, then for every $\delta>0$, at $\mathcal{H}^{m-2+\delta}$-a.e. $x\in \sing (u)$ there is a tangent function with building dimension $m-1$ (and when $m=2$ the exceptional points are in fact isolated). Such a tangent function, by the classification of $1$-dimensional $\Dir$-minimizers established in Section \ref{ss:base}, is necessarily $1$-homogeneous. Hence, setting 
\[
\sing_{m-1}(u):=\{x: I_{x,u} (0) =1\}\,, \qquad \Sigma(u) := \sing(u) \setminus \sing_{m-1}(u)
\]
the following is generally valid (even without knowing the range of the frequency function):
\[
\mathcal{H}^{m-2+\delta} (\Sigma (u) ) =0 \qquad \mbox{for $\delta >0$ and all $\Dir$-minimizers on $\Omega\subset \mathbb R^m$}\,
\]
and
\[
\Sigma (u) \qquad\mbox{is discrete in dimension $m=2$.}
\]
If in a certain dimension $m$ we know that the degree of any homogeneous $\Dir$-minimizer must be an integer, then, for any $\Dir$-minimizer $u$ on a domain $\Omega \subset \R^m$, $\{x: I_{x,u} (0)=1\}$ is relatively open in $\sing (u)$ by upper semicontinuity of the map $x\mapsto I_{x,u} (0)$. Thus by Corollary \ref{c:reg} we conclude that $\sing_{m-1}(u)$ is a $C^{1, \beta}$ submanifold. On the other hand, the same conclusion applies equally well if $u$ is defined on $\mathbb R^{m+1}$ and it is homogeneous, because in that case all tangent functions to $u$ at a point $x\neq 0$ must be invariant along the direction $\frac{x}{|x|}$ and they must therefore have the degree of an homogeneous minimizer on $\mathbb R^m$. This takes care of conclusion (b) of Theorem \ref{thm:main}.

Next we show that conclusion (c) of Theorem \ref{thm:main} holds as well. (c1) and (c2) follow from the regularity theory of \cite{DLHMS_linear}, while the Dirichlet condition \eqref{e:Dirichlet_condition} is a consequence of $\sing (u) = \{x: u (x) = Q\a{0}\}$. As for the transmission condition \eqref{e:transmission_condition}, observe that $\sing (u) \cap B_r (x)$ is $C^{1,\beta}$ and each $u_i^\pm$ is a harmonic function which vanishes on $\sing (u)$. Hence $u_i^\pm$ is in fact $C^{1,\beta}$ up to the boundary $\sing (u)$. Equation \eqref{e:transmission_condition} follows then from the inner variation identity of 
\cite[Proposition 7.1]{DLHMS_linear}. 
\end{proof}

\subsection{Proof or Theorem \ref{thm:main3}} To conclude the induction stated in Proposition \ref{p:inductive}, we only need to show that Theorem \ref{thm:main3} holds in dimension $m+1$ if it holds in dimension $m$.

\begin{proof}[Proof of Theorem \ref{thm:main3}]
Assuming therefore that Theorems \ref{thm:main}, \ref{thm:main2}, \ref{thm:main3} hold for a given dimension $m$, the aim is to show that the classification result of Theorem \ref{thm:main3} holds in dimension $m+1$. Fix therefore an $I$-homogeneous $\Dir$-minimizer $u\in W^{1,2}_{{\rm loc}} (\mathbb R^{m+1}, \Iqspecl)$. As already observed in the proof of Corollary \ref{c:reg2}, the value of the frequency function for such $u$ is an integer at every $x\in \sing (u)\setminus \{0\}$, and hence Corollary \ref{c:Lipschitz} implies its Lipschitz regularity. Next, we know from Corollary \ref{c:reg2} that $\sing_m (u)=\{x \, \colon \, I_{x,u}(0)=1\}$ is a $C^{1,\alpha}$ submanifold of $\R^{m+1}$ of dimension $m$, and that, setting as usual $\Sigma=\Sigma (u) = \sing(u) \setminus \sing_m(u)$, $\mathcal{H}^{m-1+\delta} (\Sigma) =0$ for every $\delta >0$. Consider now 
any connected component $F$ of $\partial B_1 \setminus \sing (u)$. Recalling the regularity theory, $u|_F = \left(\sum_i \a{u_i}, \epsilon (F)\right)$ for some $\epsilon (F) \in \{+1,-1\}$ and functions $u_1 \leq u_2 \leq \ldots \leq u_Q$ which are not all identically zero and restrictions to $F$ of $I$-homogeneous harmonic functions defined on the cone $\{\lambda x: x\in F, \lambda >0\}$. We therefore conclude that each (non-zero) $u_i$ is an eigenfunction of the Dirichlet Laplacian on the domain $F$. We claim that each $u_i$ does not change sign. To see this, observe first that, since $\etab \circ u = 0$ and $F \cap \sing(u) = \emptyset$, $u_1$ is strictly negative on $F$ and $u_Q$ is strictly positive on $F$. In particular, $u_1$ and $u_Q$ are eigenfunctions of the Dirichlet Laplacian on $F$ which do not change sign, and a well-known fact in spectral theory implies then that the corresponding eigenvalue is in fact the first, which has a $1$-dimensional eigenspace. There are therefore a unique function $\varphi_F$ (which for simplicity we normalize as positive) and constants $a_1(F) < 0 < a_Q(F)$ such that $u_1 = a_1\,\varphi_F$ and $u_Q=a_Q\,\varphi_F$. We next show that $u_i$ does not change sign also for $i\in \{2,\ldots,Q-1\}$, thus implying, for the same reason, that $u_i=a_i\,\varphi_F$ for some $a_i=a_i(F) \in \R$. By contradiction, let $k_1,\,k_2 \in \{2,\ldots,Q-1\}$ be, respectively, the first and the last integer so that $u_{k_j}$ changes sign, and let \[\bar u = \sum_{1\le i < k_1} u_i + \sum_{Q \ge i > k_2} u_i = \left(\sum_{i\notin\left[k_1,k_2\right]} a_i \right) \varphi_F\,.\] 
Now, if $\bar u \ge 0$ then evidently $u_{k_1}$ cannot be strictly positive at any point in $F$, for otherwise $\bfeta \circ u$ would be strictly positive there; analogously, if $\bar u \leq 0$ then $u_{k_2}$ cannot be strictly negative anywhere on $F$. Having thus proved that $u_i = a_i\,\varphi_F$ for every $i$, we next set $a(F) := \left( a_1(F), \ldots, a_Q (F) \right) \in \R^Q$, and consider the function
\[
p (x) :=
\left\{
\begin{array}{ll}
|a (F)|\varphi_F (x) \qquad &\mbox{if $x \in F$ and $\epsilon (F) =1$}\\
-|a (F)| \varphi_F (x) \qquad &\mbox{if $x\in F$ and $\epsilon (F) =-1$}\\
0&\mbox{otherwise.}
\end{array}\right.
\]
Clearly $p$ is Lipschitz and it is $I$-homogeneous. Moreover it is harmonic on $\mathbb R^{m+1} \setminus \sing (u)$. We will next show that $p$ is in fact harmonic on the whole space $\mathbb R^{m+1}$. This will imply that $p$ is a classical harmonic polynomial and thus all the claims of Theorem \ref{thm:main3} follow at once.

First of all we show that $p$ is harmonic on $\mathbb R^{m+1}\setminus \Sigma$. Consider indeed a point $x\in \sing_{m} (u)$ and, using the regularity of $\sing_{m} (u)$, fix $F^+$ and $F^-$ with $\epsilon (F^\pm)= \pm 1$ such that $x\in \partial F^+\cap \partial F^-$. Again by regularity we know that for some neighborhood $U$ of $x$, $p$ is $C^{1,\alpha}$ up to the boundary on both $U\cap F^+$ and $U\cap F^-$. Denote by $Dp^\pm (x)$ the differential on the respective sides. If $\tau$ is any direction which is tangent to $\sing_{m} (u)$, then clearly $D_\tau p^\pm (x) =0$. Let now $\nu$ be the exterior unit normal to $F^+$ at $x$. Note that \eqref{e:transmission_condition}
ensures $|\frac{\partial p^+}{\partial \nu} (x)|=|\frac{\partial p^-}{\partial \nu} (x)|$ and the definition of $p$ given above ensures that $\frac{\partial p^+}{\partial \nu} (x)$ and $\frac{\partial p^-}{\partial \nu} (x)$ have the same sign. We thus conclude that $p$ is harmonic on $\mathbb R^{m+1}\setminus \Sigma$. 

Next we conclude showing that, since $\mathcal{H}^{m} (\Sigma) = 0$ and $p$ is Lipschitz, then $p$ is harmonic on the whole space. Indeed, fix a smooth and compactly supported test function $\varphi$ and a positive number $\delta$ and let $B_{\rho_i} (x_i)$ be a finite cover of the compact set $\Sigma \cap {\rm spt}\, (\varphi)$ with the property that $\rho_i < \delta$ and 
\[
\sum_i \rho_i^m \leq \delta\, .
\] 
For each $i$ let $\chi_i$ be a bump function which vanishes on $B_{\rho_i} (x_i)$, is identically equal to $1$ on $\mathbb R^{m+1}\setminus B_{2\rho_i} (x_i)$ and satisfies the estimate
\[
0\leq \chi_i \leq 1 \qquad \mbox{and} \qquad \|\nabla \chi_i\|_{C^0}\leq \frac{C}{\rho_i}\, .
\]
Define $\varphi_\delta := \varphi \prod_i \chi_i$ and observe that
\[
\int \nabla p \cdot \nabla \varphi_\delta = 0
\] 
since $\varphi_\sigma$ is supported in $\mathbb R^{m+1}\setminus \Sigma$. Next estimate
\begin{align*}
\left|\int \nabla p \cdot (\nabla \varphi - \nabla \varphi_\delta)\right| & \leq C \|\nabla \varphi\|\|\nabla p\|_0 \sum_i \rho_i^{m+1} + \|\varphi\|_0 \|\nabla p\|_0 
\sum_i \int_{B_{2\rho_i}(x_i)} |\nabla \chi_i|\\
&\leq C \|\varphi\|_1 \|\nabla p\|_0 \sum_i (\rho_i^{m+1} + \rho_i^m) \leq C \delta \|\varphi\|_1 \|\nabla p\|_0\, .  
\end{align*}
Letting $\delta\downarrow 0$ we thus conclude that $\int \nabla \varphi \cdot \nabla p=0$ and the arbitrariness of $\varphi$ implies the harmonicity of $p$. 
\end{proof}

\newpage

\part{Nonlinear theory}

\section{Improved excess decay}\label{s:improved-decay}

In this section we prove Proposition \ref{p:decay-improved}. We first recall the following useful result.

\begin{lemma} \label{l:height bound}
Let $T$, $q$, and $\rho$ be as in Proposition \ref{p:decay-improved}. Then for every $r \leq \frac{\rho}{8}$
\begin{equation}\label{e:height bound}
\sup_{x \in \bC_{r}(q,\pi(q)) \cap \spt(T)} |\mathbf{p}_{\pi(q)^\perp}(x)| \leq C r \left( \bE^{no}(T,\bB_{8r}(q),\pi(q)) + r^2\bA^2 \right)^{\sfrac12} \,.
\end{equation}
In particular, 
\begin{equation}\label{e:inclusion}
    \bC_{r}(q,\pi(q)) \cap \spt(T) \subset \bB_{2r}(q)\,.
\end{equation}
\end{lemma}

\begin{proof}
Without loss of generality, suppose $q=0$, and let $\pi_0 = \pi(0)$. We shall simply write $\bC_r$ for $\bC_r(0,\pi_0)$. First, we recall the following $L^\infty - L^2$ estimate: 
\begin{equation}\label{Linfty-L2}
    \sup_{x\in \bC_r \cap \spt (T)} |\mathbf{p}_{\pi_0}^\perp(x)|^2 \leq C r^2 \left( r^{-(m+2)} \int_{\bC_{2r}} |\mathbf{p}_{\pi_0}^\perp(x)|^2\,d\|T\|(x) + r^2\bA^2 \right)\,.
\end{equation}
The validity of \eqref{Linfty-L2} is a simple consequence of the stationarity of the varifold $\|T\|$, so that the coordinates on the support of $T$ satisfy an elliptic PDE: the argument is due to Allard (see \cite[Theorem (6)]{Allard72}), and a proof can be found, for instance, in \cite[Lemma 1.7]{Spolaor}. Next, since $q=0$ is a point of density $Q=\frac{p}{2}$ we have by the Poincar\'e inequality (cf. e.g. \cite[Lemma 1.8]{Spolaor}) that
\begin{equation}\label{poincare}
\begin{split}
    r^{-(m+2)} \int_{\bC_{2r}} |\mathbf{p}_{\pi_0}^\perp(x)|^2\,d\|T\|(x) &\le C\,\left( \bE (T,\bC_{4r}(0,\pi_0)) + r^2\bA^2 \right)\\
    &\leq C\,\left( \bE^{no} (T,\bC_{8r}(0,\pi_0)) + r^2\bA^2 \right)\,,
    \end{split}
\end{equation}
where $\bE(T,\bC_{4r}(0,\pi_0)):= (4\omega_m r)^{-1} \|T\|(\bC_{4r}(0,\pi_0)) - Q$, and where in the last inequality we have used \cite[Theorem 16.1]{DLHMS}. The estimate \eqref{e:height bound} is then an immediate consequence of \eqref{Linfty-L2} and \eqref{poincare}. Since $8r \leq \rho$, we can then apply \eqref{e:decay} to gain 
\[
\sup_{x\in \bC_r \cap \spt (T)} |\mathbf{p}_{\pi_0^\perp}(x)| \leq C r \frac{r^{\alpha/2}}{\rho^{\alpha/2}} \left( \bE^{no}(T,\bB_\rho) + \rho^2\bA^2 \right)^{1/2} \leq C r \varepsilon_1^{1/2}\,,
\]
which implies \eqref{e:inclusion} as soon as $\eps_1$ is sufficiently small.
\end{proof}

\begin{proof}[Proof of Proposition \ref{p:decay-improved}]
We prove the decay in \eqref{e:decay-quadratic}, since \eqref{e:decay-height} is a simple consequence of the former and \eqref{e:height bound}. Thanks to \eqref{e:inclusion}, it is easy to check that 
\[
\bE^{no} (T, \bC_r (q, \pi (q))) \leq 2^m \bE^{no} (T, \bB_{2r} (q), \pi (q)) 
\leq 2^m \bE^{no} (T, \bC_{2r} (q, \pi (q)))\, .
\]
We will therefore aim at proving the decay with the cylindrical excess $\bE^{no} (T, \bC_\rho (q, \pi (q)))$ in place of the spherical excess $\bE^{no} (T, \bB_\rho (q), \pi (q))$.

It is sufficient to prove that there is a constant $\varepsilon_1$ such that, if $\bE^{no}(T,\bC_{4\rho} (q, \pi (q)))+(4\rho)^2\bA^2 < \varepsilon_1$, then either 
\begin{equation} \label{e:decay-quadratic-easy}
    \bE^{no} (T, \bC_{\rho/2} (q, \pi (q)) ) \leq \left( \frac12 \right)^{2-2\delta} \max \{\bE^{no} (T, \bC_\rho (q, \pi (q))), \varepsilon_1^{-1} \rho^2 \bA^2\}
\end{equation}
or
\begin{equation} \label{e:decay-quadratic-easy-all-the-way}
    \bE^{no} (T, \bC_{\rho/2} (q, \pi (q)) ) \leq \left( \frac18 \right)^{2-2\delta} \max \{\bE^{no} (T, \bC_{4\rho} (q, \pi (q))), \varepsilon_1^{-1} (4\rho)^2 \bA^2\}\,.
\end{equation}
Indeed, first observe that, as a consequence of \eqref{e:decay}, if $\bE^{no}(T,\bC_{4\rho} (q, \pi (q)))+(4\rho)^2\bA^2 < \bar\varepsilon_1$ for a suitable choice of $\bar\varepsilon_1$ then $\bE^{no}(T,\bC_{r} (q, \pi (q)))+r^2\bA^2 < \varepsilon_1$ for all $r \leq 4\rho$. Hence, the alternatives \eqref{e:decay-quadratic-easy}-\eqref{e:decay-quadratic-easy-all-the-way} imply that for any integers $h \geq 1$ and $k \geq 3$
\[
\begin{split}
&\max \{\bE^{no} (T, \bC_{2^{-(h+k)}\rho} (q, \pi (q))), \varepsilon_1^{-1} (2^{-(h+k)}\rho)^2 \bA^2\} \\
& \qquad \qquad \leq C(m) ( 2^{-k} )^{2-2\delta} \max \{\bE^{no} (T, \bC_{2^{-h}\rho} (q, \pi (q))), \varepsilon_1^{-1} (2^{-h}\rho)^2 \bA^2\}
\end{split}
\]
for a constant $C(m)$ independent of both $h$ and $k$. This shows the validity of the desired decay.

We will prove \eqref{e:decay-quadratic-easy}-\eqref{e:decay-quadratic-easy-all-the-way} by contradiction, and without loss of generality we assume that $q=0$ and $\rho =1$.  We assume at the same time that the tangent plane $\pi(0)$ to $T$ at $0$ is the horizontal plane $\pi_0$, which we identify with $\R^m \simeq \R^m \times \{0_n\} \subset \R^{m+n}$, and that $\R^{m+1} = \R^{m} \times \R \simeq \R^m \times \R \times \{0_{n-1}\} \subset \R^{m+n}$ is the tangent space to $\Sigma$ at $q=0$. Finally we use the notation $\bC_r$ for the cylinder $\bC_r (0, \pi_0)$.

The contradiction assumption is then that there exist $\delta > 0$ and sequences $T_k$ of currents and $\Sigma_k$ of manifolds as above with 
\begin{equation} \label{e:smallness hyp}
    \bE^{no} (T_k, \bC_4 ) + 4^2\bA_k^2 \to 0 \qquad \mbox{as $k \to \infty$}
\end{equation}
but for which
\begin{equation} \label{e:the contradiction}
     \bE^{no} (T_k, \bC_{\frac12} ) > \left( \frac12 \right)^{2-2\delta} \max\{\bE^{no} (T_k, \bC_1), k^2  \bA_k^2\}
\end{equation}
as well as
\begin{equation} \label{e:the contradiction-all-the-way}
     \bE^{no} (T_k, \bC_{\frac12} ) > \left( \frac18 \right)^{2-2\delta} \max\{\bE^{no} (T_k, \bC_4), k^2 4^2 \bA_k^2\}\,.
\end{equation}
Observe that \eqref{e:the contradiction-all-the-way} implies that
\begin{equation} \label{boundedness of excess ratios}
    \sup_k \frac{\bE^{no}(T_k,\bC_4) + 4^2 \bA_k^2}{\max\{\bE^{no}(T_k,\bC_1), k^2 \bA_k^2\}} \leq C(m)\,.
\end{equation}
We will now use \eqref{boundedness of excess ratios} to show that \eqref{e:the contradiction} leads to a contradiction. It is easy to check that, for all sufficiently large $k$, the currents satisfy all the assumptions of \cite[Theorem 16.1]{DLHMS}: in fact we only need to show that $(\mathbf{p}_{\pi_0})_\sharp T_k\res \bC_4 \equiv Q \a{B_4 (0, \pi_0)} \modp$. However, since $(\partial^p T_k) \res \bC_4 = 0$, the constancy theorem implies that $(\mathbf{p}_{\pi_0})_\sharp T_k\res \bC_4 \equiv c \a{B_4 (0, \pi_0)} \modp$ for some integer $c$, while the fact that $c$ can be taken equal to $Q$ follows from the assumption that $\Theta_{T_k} (0) =Q$ and $0$ is a flat point.

It follows from \cite[Theorem 16.1]{DLHMS} that there exist special multi-valued Lipschitz functions $u_k \colon B_{1} \subset \R^m \to \mathscr{A}_Q(\R^n)$ and closed sets $K_k \subset B_{1}$ such that, for some constants $\gamma, C$ depending only on $m$ and $p$,
\begin{itemize}
    \item[(i)] ${\rm Gr}(u_k) \subset \Sigma_k$ and $\Lip(u_k) \leq C (\bE^{no}(T_k,\bC_4) + 4^2\bA_k^2)^\gamma$;
    \item[(ii)] $|B_1 \setminus K_k| \leq \|T_k\|((B_1\setminus K_k) \times \R^n) \leq C (\bE^{no}(T_k,\bC_4) +4^2\bA_k^2)^{1+\gamma}$ and $\bG_{u_k} \mres (K_k \times \R^n) = T_k \mres (K_k \times \R^n) \; \modp$,
    \item[(iii)] for every fixed $0<r\leq 1$ the following inequality holds for $k$ large enough
    \[
    \left| \|T_k\|(\bC_r) - Q \omega_m r^m - \frac12\, \Dir(u_k, B_r) \right| \leq (\bE^{no}(T_k,\bC_4)+4^2\bA_k^2)^{1+\gamma}\,r^m \,.
    \]
\end{itemize}
Since, by (ii),
\[
\|T_k\|((B_r \cap K_k)\times \R^n) - Q |B_r \cap K_k| = \frac12 \, \int_{(B_r \cap K_k) \times \R^n} |\vec T_k - \pi_0|_{no}^2\,d\|T_k\|\,,
\]
(ii) and (iii) readily imply that for every $\frac12 \leq r \leq 1$
\begin{equation} \label{e:Dir vs Eno}
    \left| \frac{1}{2 \omega_m r^m} \Dir (u_k, B_r) - \bE^{no}(T_k, \bC_r) \right| \leq C (\bE^{no}(T_k,\bC_4)+4^2\bA_k^2)^{1+\gamma}\,.
\end{equation}
Hence, setting
\[
v_k (x) := \frac{u_k(x)}{\left( \max \{\bE^{no}(T_k,\bC_1), k^2\bA_k^2\}\right)^{\frac12}}\,,
\]
\eqref{e:the contradiction} and \eqref{boundedness of excess ratios} imply at the same time that 
\begin{equation} \label{e:contr on Dir}
    2^m \Dir(v_k,B_{\frac12}) > \left(\frac12 \right)^{2-2\delta}  \Dir(v_k,B_1) - C(m)\,(\bE^{no}(T_k,\bC_4)+4^2\bA_k^2)^\gamma\,
\end{equation}
and that 
\begin{equation}\label{e:contr on Dir 2}
\liminf_{k\to \infty} \Dir (v_k, B_{\frac{1}{2}}) >0\, .
\end{equation}

We next let $v$ denote a subsequential limit, as $k \to \infty$, of the functions $v_k$ in the weak topology of $W^{1,2}(B_1, \mathscr{A}_Q(\R^n))$. 

We rewrite the maps $u_k$ as $u_k = (\bar u_k (x), \Psi (x, \bar u_k (x)))$, where $\bar u_k$ takes values in $\mathscr{A}_Q (\mathbb R)$, while $\Psi_k : \mathbb R^{m+1}\to \mathbb R^{n-1}$ is the function whose graph gives the manifold $\Sigma_k$. Simple computations lead to 
\begin{equation} \label{e:Dirichlet comparison}
   \abs{\Dir (u_k, B_1) - \Dir (\bar u_k,B_1)} \leq C (\bA_k^2 + \bA_k^2 (\bE^{no} (T_k, \bC_1)))\,.
\end{equation}
In particular, if we set
\[
\bar v_k (x) := \frac{\bar u_k(x)}{\left( \max \{\bE^{no}(T_k,\bC_1), k^2\bA_k^2\}\right)^{\frac12}}\,,
\]
we immediately deduce that
\[
\abs{\Dir (v_k, B_1) - \Dir (\bar v_k,B_1)} \to 0
\]
and in particular that $v$ can be identified with the limit of $\bar v_k$, hence with a map taking values in $\mathscr{A}_Q (\mathbb R)$. 

Observe next that, by \cite[Theorem 13.3]{DLHMS} there is a Dir-minimizing map $h_k: B_1 \to \mathscr{A}_Q (\mathbb R)$ 
\[
\int \mathcal{G}_s (\bar v_k, h_k)^2 + \int ||D\bar v_k|-|Dh_k||^2 \to 0\, .
\]
In particular $v$ is the limit of $h_k$ and, by the compactness of Dir-minimizing maps, it is Dir-minimizing and its Dirichlet energy in $B_r$ is the limit of the Dirichlet energies in $B_r$ of $h_k$ (and hence of the Dirichlet energies of the maps $\bar v_k$) as long as $r<1$. From \eqref{e:contr on Dir} and \eqref{e:contr on Dir 2} we immediately conclude that
\begin{equation}\label{e:contr on Dir 3}
\int_{B_{1/2}} |Dv|^2 \geq 2^{-m-2+2\delta} \int_{B_1} |Dv|^2 > 0\, .
\end{equation}

Note next that, by \cite[Theorem 23.1]{DLHMS}, since $0$ is a multiplicity $Q$ point for the current $T$, we conclude that, for every fixed $\hat{\delta}$, there is a $\bar \sigma$ such that 
\[
\int_{B_\sigma} \mathcal{G}_s (v, Q \a{\bfeta\circ v})^2 \leq \hat{\delta} \sigma^m \qquad \forall \sigma \leq \bar{\sigma}\, .
\]
The latter (and the continuity of Dir-minimizers, cf. \cite{DLHMS_linear}) implies that $v (0) = Q \a{\bfeta \circ v (0)}$. By possibly subtracting a suitable constant, we can assume that indeed $v (0)= Q \a{0}$.

We next observe that, by Theorem \ref{t:uniqueness-tangent-plane} and (iii), there is a positive $\gamma$ such that, for every fixed $r \leq 1/2$, the following inequality holds provided $k$ is large enough:
\begin{equation} \label{eq:Dir decays fast}
    \Dir (u_k, B_{r}) \leq C  \Dir(u_k, B_{\frac12}) r^{m+2\gamma} + C \left(\bE^{no}(T_k,\bC_4) + 4^2\bA_k^2  \right)^{1+\gamma}\, .
\end{equation}
In particular, passing in the limit in $k$ we conclude that
\begin{equation}\label{e:Dir decays fast 2}
\Dir (v, B_r) \leq C \Dir (v, B_{\frac12}) r^{m+2\gamma}\, .
\end{equation}
The latter estimate implies that the frequency function $I_{v,0}(0) > 1$. Indeed, if this was not the case we would find a radius $t$ such that $I_{v,0} (t) \leq 1 + \frac{\gamma}{2}$, and \cite[(3.44)]{DLS_Qvfr} would imply, for $r \leq \frac{t}{2}$
\begin{equation}\label{e:C(h)}
\Dir(v,B_r) \geq C (v) \, r^{m+\gamma}\,,
\end{equation}
which is a contradiction to \eqref{e:Dir decays fast 2}. In fact \cite[(3.44)]{DLS_Qvfr} is proved in the paper for the ``standard'' multiple valued functions, but the derivation just uses the 
identities \cite[Proposition 3.2]{DLS_Qvfr}, which hold for Dir minimizing special $Q$-valued maps by \cite[Proposition 9.3]{DLHMS_linear}. Note moreover that the constant $C (v)$ in \eqref{e:C(h)} depends on $v$, but since the inequality holds for all positive radii $r$, clearly for sufficiently small ones \eqref{e:C(h)} still contradicts \eqref{e:Dir decays fast 2}. 

Next, Theorem \ref{thm:main2} guarantees that $I_{v,0}(0) \geq 2$. Again by \cite[(3.44)]{DLS_Qvfr} and the monotonicity of $r \mapsto I_{v,0}(0)$, we conclude that
\begin{equation} \label{eq:quadratic decay}
    \Dir(v, B_{r}) \leq  r^{m+2} \, \Dir(v,B_1) \qquad \mbox{for all $r \leq 1/2$}\,.
\end{equation}
In particular we have 
\[
\Dir (v, B_{1/2}) \leq 2^{-m-2} \Dir (v, B_1)\, .
\]
This however contradicts \eqref{e:contr on Dir 3} and hence completes the proof of \eqref{e:decay-quadratic} and of the proposition.
\end{proof}

\section{Refined Center Manifold and Proof of Proposition \ref{p:contact}} \label{s:supercm}

In this section we prove Proposition \ref{p:contact}. In order to do that we will first specify some additional conditions on the choice of parameters in \cite[Assumptions 17.5, 17.10, and 17.11]{DLHMS}.

\subsection{Additional conditions on the parameters}
Let $\delta_2$ be fixed as in \cite[Assumption 17.10]{DLHMS} and let 
\begin{equation}\label{eq:param1}
\delta:=\min\{\delta_2,1-4\delta_2\}\,.
\end{equation}
Let $\varepsilon_1>0$ be as in Proposition \ref{p:decay-improved} corresponding to this choice of $\delta$ and choose $\eta=\eta(m,p,N_0)>0$ such that
\begin{equation}\label{eq:param2}
\eta=c_g\, 2^{-N_0}\,,
\end{equation}
where $c_g>0$ is a geometric constant to be defined later and $N_0$ is the parameter in \cite[Assumption 17.11]{DLHMS}.

\begin{lemma}[Geometric decay in cubes]\label{lem:geom_decay}
Assume that $T, \Sigma$, and the parameters $\beta_2, \delta_2, M_0$ satisfy \cite[Assumptions 17.5, 17.10, and 17.11]{DLHMS}. There exist  positive constants $c_g=c_g(m,p)$, $N_0=N_0(m,p,M_0)$, $C_g=C_g(N_0,m,p)$ with the following property. Assume that $C_h$ and $C_e$ satisfy \cite[Assumption 17.11]{DLHMS}, and let, in addition, $\delta$ be as in \eqref{eq:param1}. Following the Whitney decomposition of \cite[Definition 17.12]{DLHMS} assume that for some cube $L\in \sC^j$, $j\geq N_0$, there exists $q_0=(x_0,y_0)\in \sing_f(T)\cap \bB_\eta$ with 
\begin{equation}\label{e:corona}
|x_0-x_L|\leq 128\,\sqrt{m}\,\ell(L)\,\,.
\end{equation}
Then there exists $\eps_2=\eps_2(\delta, N_0,m,p)$ such that, following the notation of \cite[Eq. (17.8)]{DLHMS}, if $\bmo\leq \eps_2^2$ and $\eta=c_g2^{-N_0}$, then
\begin{gather}
\bE^{no} (T, \B_L) \leq C_g \, \bmo\, \ell (L)^{2-2\delta_2}\,,\label{eq:geometric_decay1}\\
\bh (T, \B_L) \leq C_g \,\bmo^{\sfrac{1}{2}} \ell (L)^{1+\beta_2} \,. \label{eq:geometric_decay2}  
\end{gather}
\end{lemma}

\begin{proof} Let $\varepsilon_2$ to be suitably chosen later. We notice that, by \cite[Eq. (17.8)]{DLHMS},
\[
    \bE^{no}(T, \bB_{6\sqrt{m}}, \pi_0)+ (6\sqrt{m})^2\,\bA^2< \bmo \leq \eps_2^2 \,, 
\]
and moreover, by \cite[Lemma 17.8]{DLHMS},
\begin{equation}\label{e:preheight2}
\bh(T, \bC_{5\sqrt{m}}(0,\pi_0))\leq C_0\, \bmo^{\sfrac1{2m}} \,,
\end{equation}
so that, upon choosing $\varepsilon_2$ sufficiently small, the assumptions of Proposition \ref{p:decay-improved} are satisfied with $4\rho=6\sqrt{m}$ and we have for every $0<r\leq \sqrt{m}/8$
\begin{align}
&\bE^{no} (T, \bB_r, \pi_0) \leq \bE^{no} (T, \bC_r (0, \pi_0)) \leq C \,r^{2-2\delta}\bmo\, .
\end{align}
Then observe that, by our choice of $\eta$, we have that
\begin{align*} 
\bE^{no} (T, \B_{C_0\,2^{-N_0}}(q))
        &\leq  \bE^{no} (T, \B_{C_0\,2^{-N_0}}(q), \pi_0)\\
        &\leq C_0^{-m}(6\sqrt{m}2^{N_0})^{m} \bE^{no} (T, \bB_{6\sqrt{m}},\pi_0)\leq C_g\, \bmo\, \quad \mbox{for every $q \in \bB_\eta$},
\end{align*}
where $C_g=C_g(m,p,N_0)$ will change from line to line in the coming estimates. This implies that, up to choosing $\eps_2$ sufficiently small, depending on $N_0,m,p$, the assumptions of Theorem \ref{t:uniqueness-tangent-plane} are satisfied in $\B_{C_02^{-N_0}}(q)$ for $q \in \bB_\eta \cap \Sing_f (T)$, and so we can estimate, with the usual notation for $\pi(q)$,
\[
\bE^{no}(T,\B_{2^{-1}C_02^{-N_0}}(q), \pi(q))\leq C\, \left(\bE^{no} (T, \B_{C_0\,2^{-N_0}}(q))+(6\sqrt{m})^2\,\bA^2\right) \\ 
\leq C_g\,\bmo\,,
\]
so that
\begin{align}\label{eq:planetilt}
 |\pi(q)-\pi_0|_{no}^2 \leq C \left(\bE^{no} (T, \B_{C_0\,2^{-N_0}}(q), \pi(q))+\bE^{no} (T, \B_{C_0\,2^{-N_0}}(q), \pi_0) \right) 
 \le C_g\,\bmo\,.
\end{align}
Combining \eqref{e:preheight2} with this estimate we obtain
\begin{equation}\label{e:preheight3}
\bh(T, \bC_{C_0\,2^{-N_0}}(q,\pi(q)))\leq \bh(T, \bC_{5\sqrt{m}}(0,\pi_0))+ C\,|\pi_0-\pi(q)|_{no}\leq C_g\, \bmo^{\sfrac1{2m}}\,.
\end{equation}
Choosing $\eps_2=\eps_2(N_0, \delta,m,p)$ sufficiently small, we can then apply Proposition \ref{p:decay-improved} at every point $q\in \sing_f\cap \bB_\eta$ to obtain, for every $0<r<\sfrac{1}{32} \,C_0\,2^{-N_0}$,
\begin{gather}
    \bE^{no}(T, \bB_r(q),\pi(q))\leq \bE^{no}(T, \bC_r(q,\pi(q)))\leq C_g\,\bmo\,r^{2-2\delta}\,,\label{eq:dec_exc_atq}\\
    \bh(T, \bB_r(q),\pi(q))\leq \bh(T, \bC_r(q,\pi(q)))\leq C_g\,\bmo^{\sfrac12}\,r^{2-\delta}\,.\label{eq:height_exc_atq}
\end{gather}
In particular, combining \eqref{eq:height_exc_atq} and \eqref{eq:planetilt}, for $p=(x,y) \in \spt(T)\cap \B_{C_02^{-N_0}}(q)$ and $q=(x_q,y_q)$, we get
\begin{align*} 
|y-y_q|
    &\le |\p^\perp_{\pi_0}(p-q)|\le
    |\p^\perp_{\pi(q)}(p-q)|+|\pi_0-\pi(q)|_{no}\,|p-q|\\
    &\le C_g\,\bmo^{\sfrac12}\, |p-q|
\end{align*}
that is choosing $\eps_2$ sufficiently small
\begin{equation}\label{eq:key_est}
  |y-y_q|\leq C_g \bmo^{\sfrac12} |x-x_q|\qquad \forall p=(x,y) \in \spt(T)\cap \B_{C_02^{-N_0}}(q). 
\end{equation}

The conclusion of the proof is a straightforward consequence of \eqref{eq:dec_exc_atq}, \eqref{eq:height_exc_atq} and \eqref{eq:key_est}. Indeed combining \eqref{eq:key_est}, \eqref{e:corona} and the fact that $q\in \bB_\eta$, together with our choice of $\eta$, we have
\begin{equation}\label{eq:key_cont}
    \bB_L\subset \bB_{C_0\ell(L)}(q_0)\,.
\end{equation}
Then, using \eqref{eq:dec_exc_atq}, we have
\[
\bE^{no}(T,\bB_L) \le C_0^m \bE^{no}(T,\bB_{C_0\ell(L)}(q_0)) \le C_g\,\bmo\, \ell(L)^{2-2\delta}\,.
\]
Moreover, using this we have
\[
|\pi_L-\pi(q_0)|_{no}^2 \le C_0 (\bE^{no}(T,\bB_L) + \bE^{no}(T,\bB_L, \pi(q_0))) \le C_g\,\bmo\, \ell(L)^{2-2\delta}\,\]
so that, from \eqref{eq:height_exc_atq} we conclude
\[
\bh(T,\bB_L)\le \bh(T,\bB_{C_0\ell(L)}(q_0),\pi(q_0)) + C\,\ell(L)\,|\pi_L - \pi(q_0)|_{no}\le C_g\,\bmo^{\sfrac12}\, \ell(L)^{2-\delta}\,.
\]
In particular \eqref{eq:geometric_decay1} and \eqref{eq:geometric_decay2} are satisfied.
\end{proof}

Now we can specify the additional conditions on the parameters $N_0, C_e, C_h, \eps_2$ in the construction of the center manifold which are needed to prove Proposition \ref{p:contact}.

\begin{ipotesi}\label{a:parameters_fine_cm}
We assume that
\begin{enumerate}
    \item $N_0$ is larger than $C(\beta_2,\delta_2,M_0)$ as in \cite[Assumption 17.11]{DLHMS} and as in Lemma \ref{lem:geom_decay}.
    \item $C_e$ is larger than $C(\beta_2,\delta_2,M_0,N_0)$ as in \cite[Assumption 17.11]{DLHMS} and than $C_g$ in Lemma \ref{lem:geom_decay}.
    \item $C_h$ is larger than $C(\beta_2,\delta_2,M_0,N_0,C_e)$ as in \cite[Assumption 17.11]{DLHMS} and than $C_g$ in Lemma \ref{lem:geom_decay}.
\item  $\eps_2$ is smaller than a positive constant $c(\beta_2, \delta_2, M_0, N_0, C_e, C_h)$ and as in Lemma \ref{lem:geom_decay}.
\end{enumerate}
\end{ipotesi}

\subsection{Proof or Proposition \ref{p:contact}}

Proposition \ref{p:contact} is a straightforward consequence of our choice of the parameters in Assumption \ref{a:parameters_fine_cm} and Lemma \ref{lem:geom_decay}.

\medskip

To prove \eqref{e:fine_cm1}, suppose by contradiction that there exists $q_0=(x_0,y_0)\in \sing_f (T)\cap\bB_\eta$ such that $q\notin \Phi(\Gamma)$. Then, by definition of $\Gamma$ in the Whitney refining procedure of \cite[Definition 17.12]{DLHMS}, there exist $j\geq N_0$ and $L \in \mathscr{C}^j$ such that \eqref{e:corona} holds and
\[
L\in \sW^j=\sW_e^j\cup \sW_h^j\cup \sW_n^j \,.
\]
If $L\in \sW_e^j\cup \sW_h^j$, then the assumptions of Lemma \ref{lem:geom_decay} are clearly satisfied by $L$ and so either \eqref{eq:geometric_decay1} or \eqref{eq:geometric_decay2} contradict the stopping assumption for cubes in $W^j_e$ or $W_h^j$ respectively. If $L\in \sW^j_n$, then there must exist an ancestor $\tilde L \in W^k_e\cup W^k_h$, $N_0\leq k<j$. Then 
$$
|x_{\tilde L}-x_0|\leq 4\sqrt{m} \sum_{i=k}^j {2^{-i}}\leq 4\sqrt{m} \,\ell(\tilde L)\,
$$
so that $\tilde L$ again satisfies the assumptions of Lemma \ref{lem:geom_decay}, thus a contradiction.

\medskip

To prove \eqref{e:fine_cm2}, we assume by contradiction that there is $L\in \sW^j$ such that $\ell(L)>\frac{1}{64\sqrt{m}}\,\dist(x_0,L)$, then we have
\[
|x_L-x_0|\leq \dist(x_q,L)+2\sqrt{m}\ell(L)\leq 128\sqrt{m}\ell(L)\,,
\]
so that we are in the same assumption as in the previous part of the proof and the contradiction follows in the same way. \qed

\section{Almost monotonicity of the frequency function and blow-up} \label{s:monot}

\subsection{Proof of Proposition \ref{p:almost-monot}}
In this section we prove Proposition \ref{p:almost-monot}. The main point is that, thanks to \eqref{e:fine_cm2}, the estimates of Proposition \cite[Proposition 26.4]{DLHMS} apply to $\bH (q,r)$ and $\bD (q, r)$ whenever $q\in {\rm Sing}_f (T) \cap \bB_\eta$ and $r\leq 1$. Indeed, \eqref{e:fine_cm2} corresponds to \cite[Eq. (25.5)]{DLHMS}, which is the only condition on the radius $r$ used in the proof of \cite[Proposition 26.4]{DLHMS}. We report here the relevant estimates, and in order to do so we need to introduce the following additional quantities.

\begin{definition}\label{d:funzioni_ausiliarie}
Fix $q\in \mathcal{M}$.
We let $\partial_{\hat r}$ denote the derivative with respect to arclength along geodesics emanting from $q$ and we set
\begin{align}
&\bE (q,r) := - \int_\cM \phi'\left(\textstyle{\frac{d(p,q)}{r}}\right)\,\sum_{i=1}^Q \langle
N_i(p), \partial_{\hat r} N_i (p)\rangle\, dp\,  ,\\
&\bG (q, r) := - \int_{\cM} \phi'\left(\textstyle{\frac{d(p,q)}{r}}\right)\,d(p,q) \left|\partial_{\hat r} N (p)\right|^2\, dp\,,\\
\quad\mbox{and}\quad
&\bSigma (q,r) :=\int_\cM \phi\left(\textstyle{\frac{d(p,q)}{r}}\right)\, |N|^2(p)\, dp\, .
\end{align}
\end{definition}

We are now ready to state the counterpart of \cite[Proposition 26.4]{DLHMS} (cf. also \cite[Proposition 3.5]{DLS_Blowup}). For reasons which will become clear later, we need some additional estimates and some refined assumptions. 

\begin{proposition}[First variation estimates]\label{p:variation}
For every $\gamma_3$ sufficiently small and every positive $c_0\leq 1$ there are constants $C = C (\gamma_3, c_0)>0$ and $\varepsilon_3 (\gamma_3)>0$ with the following properties. First of all, the following estimate holds for every $q\in \mathcal{M}\cap \bB_1$ and every $r\in ]0,1]$:
\begin{equation}\label{e:H'bis}  
\left|\bH' (q,r) - \textstyle{\frac{m-1}{r}}\, \bH (q,r) - \textstyle{\frac{2}{r}}\,\bE(q,r)\right|\leq  C \bH (q,r)\, .
\end{equation}
Next the following inequalities hold for every $q\in {\rm Sing}_f (T)\cap \bB_\eta$ and a.e. $r\in ]0,1]$ such that $\bI (q,r) \geq c_0$:
\begin{align}
\left|\bD (q,r)  - r^{-1} \bE (q,r)\right| &\leq C \bD (q,r)^{1+\gamma_3} + C \eps_3^2 \,\bSigma (q,r),\label{e:out}\\
\left| \bD'(q,r) - \textstyle{\frac{m-2}{r}}\, \bD(q,r) - \textstyle{\frac{2}{r^2}}\,\bG (q,r)\right| &\leq
C \bD (q,r)+ C \bD (q,r)^{\gamma_3} \bD' (q,r)\nonumber\\
&\qquad + C r^{-1}\bD(q,r)^{1+\gamma_3},\label{e:in}\\
\bSigma (q,r) +r\,\bSigma'(q,r) &\leq C  \, r^2\, \bD (q,r)\, \leq C r^{2+m} \eps_3^{2}\,.\label{e:Sigma1}
\end{align}
Finally, the following inequality holds for every $q\in {\rm Sing}_f (T)\cap \bB_\eta$ and a.e. $r\in ]0,1]$ such that $\bI (q,r) \leq c_0^{-1}$:
\begin{equation}\label{e:out-with-H}
\left|\bD (q,r)  - r^{-1} \bE (q,r)\right| \leq C (r^{-1} \bH (q,r))^{1+\gamma_3} + C \eps_3^2 \,r \bH (q,r)\, .
\end{equation}
\end{proposition}

\begin{proof} The inequality \eqref{e:H'bis} is stated in \cite[Proposition 26.4]{DLHMS} under the additional assumption that $q=0$ and that the radii $r$ satisfy $\bI (0,r) \geq 1$. However it can be easily seen that the proof given in \cite[Section 3.1]{DLS_Blowup} is valid for every $q\in \bB_1$ and every radius $r$: it is written for the case of $N$ being classical multivalued, but the case of $\mathscr{A}_Q$-valued $N$ is a routine modification. Next, the inequalities \eqref{e:out}, \eqref{e:in}, and \eqref{e:Sigma1} are claimed in \cite[Proposition 26.4]{DLHMS} under the assumption $q=0$ and the further restriction that $\bI (0,r) \geq 1$. The latter assumption is just needed to bound $\bH (q,r)$ with $r \bD (q,r)$ and the weaker version $\bI (q,r) \geq c_0$ will simply give $\bH (q,r) \leq c_0^{-1} r \bD (q,r)$, which in turn will just imply the desired inequalities with a constant $C$ which depends on $c_0$ as well. As for substituting $0$ with any $q\in \bB_\eta$, the proof of \eqref{e:Sigma1} is in fact valid for every $q$ (cf. \cite[Lemma 3.6]{DLS_Blowup}) while the arguments needed to show \eqref{e:out} and \eqref{e:in} given in \cite[Section 4]{DLS_Blowup} only use the condition \eqref{e:fine_cm2}, which is valid for $q\in {\rm Sing}_f (T)\cap\bB_\eta$. 

We finally come to \eqref{e:out-with-H}. Observe first that, by \cite[Lemma 3.6]{DLS_Blowup}, 
\begin{equation}\label{e:use-H}
\bSigma (q,r) \leq C_0 r^2 \bD(q,r) + C_0 r \bH (q,r) \leq C c_0^{-1} r\bH (q,r)\, ,
\end{equation}
for some geometric constant $C_0$, where in the last inequality we have used $\bI (q,r) \leq c_0^{-1}$ (again the proof in \cite{DLS_Blowup} is given for $\mathcal{A}_Q$-valued maps, but the $\mathscr{A}_Q$-valued case requires minor adjustments). Next, the arguments given in \cite[Section 4.3]{DLS_Blowup} imply in fact
\begin{equation}\label{e:more-general}
\left|\bD (q,r)  - r^{-1} \bE (q,r)\right| \leq C (\bD (q,r) + \bSigma (q,r))^{1+\gamma_3} + C \eps_3^2 \bSigma (q,r)\, ,
\end{equation}
from which \eqref{e:out-with-H} will follow using \eqref{e:use-H}.
As for the validity of \eqref{e:more-general}, we can just use the arguments in \cite[Section 4.3]{DLS_Blowup} after showing that, when dropping any control on $\bI (q,r)$ (but keeping the information \eqref{e:fine_cm2}), \cite[Eq. (4.9)]{DLS_Blowup} and \cite[Eq. (4.11)]{DLS_Blowup} can in fact be substituted with the more general 
\begin{align}
\sum_i \left(\inf_{\mathcal{B}^i} \mathbf{m}_0 \varphi_r\right) \ell_i^{m+2+\gamma_2/4} &\leq C_0 (\bD (q,r) + \bSigma (q,r))\\
\sup_i \mathbf{m}_0^t \left[\ell_i^t +\left(\inf_{\mathcal{B}^i} \varphi_r\right)^{t/2}\ell_i^{t/2}\right]
&\leq C(t) \bD(q,r)^{\gamma (t)}\, .
\end{align}
A simple inspection of the proof given in \cite[Lemma 4.5]{DLS_Blowup} shows that it indeed gives the latter two more general inequalities (cf. \cite[Eq. (4.12) and Eq. (4.13)]{DLS_Blowup}), while the assumption $\bI (q,r)\geq 1$ is only used at the end to bound the term $\bSigma (q,r)$ with $r^2 \bD (q,r)$. 
\end{proof}

\begin{lemma}\label{l:frequency well posedness}
There exist $0<\bar \eta<\eta(p,m)$, $\bar r > 0$, and $\bar C > 0$ such that $\mathbf{I}(q,r) \leq \bar C$ for every $q \in \mathbf{S}_f = \Sing_f(T) \cap \overline{\bB}_{\bar\eta}$ and for every $0 < r \leq \bar r$.
\end{lemma}

\begin{proof}
\textit{Step I:} By the same proof as in \cite[Theorem 3.2]{DLS_Blowup}, we observe that there exists a geometric constant $C_0>0$ such that for every $q \in \Sing_f (T) \cap \bB_\eta \subset \mathcal M$ and for every $[a,b]\subset [0, 1]$ with $\bH|_{[a,b]}>0$ we have
\begin{equation}\label{e:freq_upper_bd}
    \bI(q,a)\leq C_0(1+ \bI(q,b))\,,
\end{equation}
so that in particular 
\begin{equation}\label{e:h_up_bd}
    a\,\bD(q,a)\leq C_0(1+ \bI(q,b))\,\bH(q,a)\,.
\end{equation}

\medskip

\textit{Step II:} We next claim that for every $q \in \Sing_f (T) \cap \bB_{\eta}$ there exists $r_q >0 $ such that $\bI(q,r)$ is well defined for every $0<r<r_q$. Indeed, first observe that there exists $0<r_q<1$ such that $\bH(q,r_q)>0$, otherwise $\bH(q,r)=0$ for every $r\in (0,1)$ and therefore $T\res \bB_1(q)=Q\a{\cM}$ so that $q$ would be a regular point, which is a contradiction. Analogously, we have that if $\bH(q,s)>0$ then
\begin{equation}\label{e:freq_well_def}
    \bH(q,r)>0\qquad \forall r\in (0, s]\,,
\end{equation}
since, if not, we can let $r_*$ be the largest radius smaller than $s$ such that $\bH(q,r_*)=0$, and then by \eqref{e:h_up_bd} we would have
\[
r_*\,\bD(q,r_*)\leq C_0(1+ \bI(q,s))\,\bH(q,r_*)=0
\]
so that once again $T\res \bB_{r_*}(q)=Q\a{\cM}$, a contradiction.

\textit{Step III:} From the continuity of $q \mapsto \bH(q,r_0)$ we deduce that there exists $\bar \eta>0$ such that $\bH(q,r_0)\geq \bH(0,r_0)/2>0$ for every $q\in \Sing_f(T) \cap \overline{\bB}_{\bar\eta}$ and therefore by Step II we have that for every $q\in\Sing_f(T) \cap \overline{\bB}_{\bar\eta}$  
\begin{equation}\label{e:freq_well_def_from_zero}
    \bH(q,r)>0\qquad \forall r\in (0, r_0]\,,
\end{equation}
Applying \eqref{e:freq_upper_bd} with $b=\bar r:=r_0$, we then have
\begin{equation}
    \bI(q,r)\leq C_0(1+ \bI(q,\bar r))= C_0\left(1+ \frac{\bar r \, \bD(q,\bar r)}{\bH(q,\bar r)}\right)\leq C_0\left(1+ \frac{2\bar r\,\bD(0,2(\bar r+\bar\eta))}{\bH(0,\bar r)}\right)
\end{equation}
for every $q\in\Sing_f(T) \cap \overline{\bB}_{\bar\eta}$ and for every $0<r\leq \bar r$.
\end{proof}

We next follow the computations in \cite{DLHMS} (cf. \cite{DLS_Blowup}) and make them slightly more precise to prove the following.

\begin{lemma}\label{l:almost_monot}
There is a constant $\gamma \in \left(0,1\right)$ with the following property. For every $c_0>0$ there is a positive $C$ depending on $c_0$ and on the constant $\bar C$ of Lemma \ref{l:frequency well posedness} such that the following holds. There exist functions $\Lambda=\Lambda(q,r)$ and $\Xi = \Xi(q,r)$ such that $0<\Lambda(q,r) \leq C\,r^\gamma$ and $0<\Xi(q,r) \leq C\, \bD^\gamma(q,r)$ on $\mathbf{S}_f \times \left(0,\bar r \right]$ and
\begin{equation} \label{almost monotonicity wlb}
\frac{d}{dr} \Big(\exp(\Lambda(q,r)) \bI (q,r) + \Xi(q,r)\Big) \geq 0   
\end{equation}
for every $q\in \mathbf{S}_f$ and for every $0 < r \leq \bar r$ such that $\bI (q,r) \geq c_0$.
\end{lemma}

\begin{proof}
In order to ease the notation, we drop the dependence on $q \in \mathbf{S}_f$ from all quantities. We compute, for a.e. $r \in \left( 0, \bar r \right]$ such that $\bI(r) \geq c_0$: 
\begin{align*}
    \bI'(r)
        &=\frac{r\,\bD'(r)}{\bH(r)}+\frac{\bD(r)}{\bH(r)}-\frac{r\,\bD(r)\,\bH'(r)}{\bH^2(r)}\\
        &\stackrel{\eqref{e:H'bis} \& \eqref{e:in}}{\geq}\frac{r}{\bH(r)}\left(\frac{m-2}{r}\bD(r)+\frac2{r^2}\bG(r)-C\,\bD(r)-C\,\bD^{\gamma_3}(r)\,\bD'(r)-\frac{C}r \bD^{1+\gamma_3}(r) \right)\\
        &\qquad\qquad+\frac{\bD(r)}{\bH(r)}-\frac{r\,\bD(r)}{\bH^2(r)}\,\left(\frac{m-1}{r}\,\bH(r)+\frac2{r}\bE(r)+C\,\bH(r)\right)\\
        &= \frac{2}{r}\left(\frac{\bG(r)}{\bH(r)}-\frac{r\,\bE(r)\,\bD(r)}{\bH^2(r)} \right) -C\,\bI(r)-C\,\frac{\bD^{\gamma_3}(r)}{r}\,\bI(r)-C\, \frac{r\,\bD^{\gamma_3}(r)\,\bD'(r)}{\bH(r)}\\
        & \stackrel{\eqref{e:out-with-H}}{\geq} \frac{2}{r}\left(\frac{\bG(r)}{\bH(r)}-\frac{\bE^2(r)}{\bH^2(r)} \right) -C \frac{\bE(r)\,\bD^{\gamma_3}(r)}{r\,\bH(r)}-C\,\frac{\bE(r)}{\bH(r)} \\
       & \qquad\qquad-C\,\bI(r)-C\,\frac{\bD^{\gamma_3}(r)}{r}\,\bI(r)-C\, \frac{r\,\bD^{\gamma_3}(r)\,\bD'(r)}{\bH(r)}\\
        &\geq -C \frac{\bE(r)\,\bD^{\gamma_3}(r)}{r\,\bH(r)}-C\,\frac{\bE(r)}{\bH(r)} -C\,\bI(r)-C\,\frac{\bD^{\gamma_3}(r)}{r}\,\bI(r)-C\, \frac{r\,\bD^{\gamma_3}(r)\,\bD'(r)}{\bH(r)}\\
        &\geq -C\,\left(1+\frac{\bD^{\gamma_3}(r)}{r}\right)\,\bI(r)-C\, \bD^{\gamma_3-1}(r)\,\bD'(r)\,,
\end{align*}
where in the second inequality we used \eqref{e:out-with-H} and $\bI(r) \geq c_0$ to estimate $r^{-1}\bH(r) \leq c_0^{-1}\,\bD(r)$, in the third inequality we used the Cauchy-Schwarz inequality, and in the fourth inequality we used \eqref{e:out} and \eqref{e:Sigma1} to estimate
$$
\frac12 \,\bD(r)\leq \frac{\bE(r)}{r}\leq 2\,\bD(r)\,.
$$

Setting
\[
\lambda(r) := C\, \left( 1 + \frac{\bD^{\gamma_3}(r)}{r} \right)\,, \qquad \xi(r) := C\,\bD^{\gamma_3-1}(r) \,\bD'(r)\,,
\]
we have then showed that $\bI(r)$ satisfies, at almost all radii $r$ where $\bI(r) \geq c_0$, the differential inequality
\[
\bI' \geq -\lambda\,\bI - \xi\,,
\]
so that we immediately conclude that, at such values of $r$:
\[
\frac{d}{dr} \left( \exp\left( \int_0^r \lambda(s)\,ds \right) \, \bI(r) + \int_0^r \xi(\tau) \,\exp\left( \int_0^\tau \lambda(s)\,ds \right)  \,d\tau  \right) \geq 0\,.
\]

The conclusion follows setting
\[
\Lambda(r) := \int_0^r \lambda(s)\,ds\,, \qquad \qquad \Xi(r) := \int_0^r \xi(\tau) \,\exp(\Lambda(\tau))  \,d\tau\,,
\]
and observing that, since $\bD(s) \leq \varepsilon_3^2\,s^m \leq s$ (see e.g. \eqref{e:Sigma1}) we have
\[
\Lambda(r) = C\, \int_0^r \left(1+\frac{\bD^{\gamma_3}(s)}{s}\right)\,ds\leq C\,r^{\gamma_3}\,,
\]
and
\[
\Xi (r) \leq C\,\int_0^r \bD^{\gamma_3-1}(s)\,\bD'(s)\, ds = C\, \bD^{\gamma_3}(r)\,.\qedhere
\]
\end{proof}

In order to complete the proof of Proposition \ref{p:almost-monot}, it suffices to show that $\bI (q,r) \geq c_0$ for some positive constant $c_0$. This will be accomplished in the following lemma. 

\begin{lemma}\label{l:remove_assumption}
There is $r_0 >0$ with the property that $\bI (q, r)\geq \frac{5}{4}$ for every $q \in \mathbf{S}_f$ and every $0<r\leq r_0$.
\end{lemma}
\begin{proof} In order to simplify our notation we drop the dependence on $q \in \mathbf{S}_f$. Consider the function $\bOmega (r):= \max \left\{\bI (r), \frac{1}{2}\right\}$. We notice first that, by standard measure theory, $\bOmega$ has derivative $\bOmega'(r)=0$ for a.e. $r$ such that $\bI(r) = \frac12$. By Lemma \ref{l:almost_monot} we have
\[
\frac{d}{dr} \Big(\exp(\Lambda(r))\bOmega (r) + \Xi(r)\Big)\geq 0 \qquad \mbox{on $\left\{\bI > \frac{1}{2}\right\}$}\,,
\]
whereas a.e. on $\left\{\bI \leq \frac12\right\}$ it holds
\[
\frac{d}{dr} \Big(\exp(\Lambda(r))\bOmega (r) + \Xi(r)\Big) = \Lambda'(r) \exp(\Lambda(r)) \bOmega(r) + \Xi'(r) = \exp(\Lambda(r)) \, \left( \lambda(r) \bOmega(r) + \xi(r) \right) \geq 0\,.
\]
In particular we easily conclude from the properties of $\Lambda$ and $\Xi$ that, if $r_0$ is chosen sufficiently small, then 
\[
\bOmega (\sigma) \leq \frac{9}{8} \bOmega (s) + \frac{1}{32} \qquad \forall \sigma \leq s \leq r_0\, .
\]
Thus, if for some $s\leq r_0$ we have $\bI (s)\leq \frac{5}{4}$, then $\bI (\sigma)\leq \bOmega (\sigma) \leq \frac{46}{32}$ for every $0<\sigma \leq s$.

We now compute 
\[
\frac{d}{dr} \left( \ln \frac{\bH (r)}{r^{m-1}} \right) = \frac{\bH' (r)}{\bH (r)} - \frac{m-1}{r} = \frac{2}{r} \frac{\bE (r)}{\bH (r)} + {\rm O}(1) = \frac{2\bI(r)}{r} + {\rm O}(r^{\gamma-1})
\]
where in the second identity we have used \eqref{e:H'bis}, and in the third identity we have used \eqref{e:out-with-H} together with the estimate $\bH(r) \leq C\, r^{m+1}$. For $r\in ]0,s]$, using $\bI (r) \leq \frac{46}{32}$ we then conclude $\frac{d}{dr} \left(\ln \frac{\bH (r)}{r^{m-1}}\right) \leq \frac{46}{16r} + C r^{\gamma-1}$ and hence, choosing $s_1\leq s$ sufficiently small we can estimate
\[
\frac{d}{dr} \left( \ln \frac{\bH (r)}{r^{m-1}}\right) \leq \frac{31}{8r} \qquad \forall 0<r\leq s_1\,.
\]
Integrating the latter inequality between $r$ and $s_1$ we conclude 
\[
\frac{\bH (s_1)}{s^{m+3-1/8}_1} \leq \frac{\bH (r)}{r^{m+3-1/8}}\, .
\]
In particular we infer 
\begin{equation}\label{e:from-above}
\bH (r) \geq c\, r^{m+3-1/8}
\end{equation} 
for some positive constant $c$ and every sufficiently small $r$. Note however that, by the decay in Proposition \ref{p:decay-improved}, it is easy to see that for every positive $\delta>0$ and every $q\in \mathbf{S}_f$ there is a constant $C(\delta)$ such that 
\[
\|N\|_{L^\infty (\mathcal{B}_r (q))} \leq C r^{2-\delta}\, .
\]
In particular we conclude 
\begin{equation}\label{e:from-below}
\bH (r) \leq C r^{m+3-\delta}\, .
\end{equation}
Since \eqref{e:from-above} and \eqref{e:from-below} are not compatible, we conclude that the premise, i.e. the existence of an $s\leq r_0$ at which $\bI (s)\leq \frac{5}{4}$, is incorrect. 
\end{proof}

\subsection{Proof of Proposition \ref{p:blow-up}} We are now ready to prove Proposition \ref{p:blow-up}. First of all we remark that the conclusions (i) and (ii) follow immediately from \cite[Theorem 28.2]{DLHMS}. The only differences that must be taken into account are the following:
\begin{itemize}
    \item The maps $N^b_k$ examined in \cite[Theorem 28.2]{DLHMS} are defined over possibly different tangent spaces $T_{q_k} \mathcal{M}_k$ due to the fact that the center manifolds might actually change in that case. However the situation that all points $q_k$ coincide with a single point $q$ and all center manifolds $\mathcal{M}_k$ coincide with a single center manifold $\mathcal{M}$ is included in the statement of \cite[Theorem 28.2]{DLHMS}, so that our situation is just a particular case. \item The normalization of the maps $N^b_k$ is different and it is given in \cite[Theorem 28.2]{DLHMS} by
\[
N^b_k (x) := \frac{N (\mathbf{e} (q, r_k x))}{\left( r_k^{1-m} \mathbf{H} (q,r)\right)^{\sfrac{1}{2}}}\, . 
\]
Since however Proposition \ref{p:almost-monot} implies that the limit
\[
\lim_{k\to\infty} \frac{r_k^{2-m} \mathbf{D} (q,r)}{r_k^{1-m} \bH (q,r)}
\]
exists and it is finite and positive, the limit $N^b_\infty$ of the $N^b_k$ and the limit $\bar N$ of the $N_{q,r_k}$ differ only by a positive scaling factor. 
\item The strong convergence in $W^{1,2}$ is not stated in \cite[Theorem, 28.2]{DLHMS}, but however it is a direct consequence of the argument given for its $\Dir$-minimality, which shows that 
\[
\Dir(f,B_{3/2}) \geq \limsup_{k \to \infty} \Dir(N^b_k,B_{3/2})
\]
for every competitor $f\in W^{1,2} (B_{3/2})$ with $f|_{\partial B_{3/2}}= N^b_\infty|_{\partial B_{3/2}}$. Since we can use directly $f= \bar N$ in the latter comparison, we conclude that the Dirichlet energies of $N_{q,r_k}$ converge to the Dirichlet energy of $\bar N$, which in turn implies the strong $W^{1,2}$ convergence by standard arguments. 
\end{itemize}
Observe next that the strong $W^{1,2}$-convergence implies as well that 
\[
\frac{\rho \int |D\bar N|^2 \phi (\rho^{-1} |x|)}{- \int |x|^{-1} \phi' (\rho^{-1} |x|) |\bar N|^2} = \bI (q,0) \qquad \mbox{for all $\rho\in ]0,1[$}\,.
\]
In particular by \cite[Theorem 9.2]{DLHMS_linear} we infer (iii). Finally (iv) is a consequence of Theorem \ref{thm:linear-what-is-needed} and the lower bound $\bI (q,0)\geq \frac{5}{4}$ given by Lemma \ref{l:remove_assumption}. \qed

\section{Conclusion} \label{s:conclusion}

In this section we prove Lemma \ref{l:closed} and Proposition \ref{p:reduction_argument}, thereby concluding the proofs of Theorem \ref{t:even} and Theorem \ref{t:even-structure}. 

\subsection{Proof of Lemma \ref{l:closed}} Observe that, by Theorem \ref{t:uniqueness-tangent-plane}, it follows immediately that there are positive geometric constants $\varepsilon$ and $C$ such that, if $\bE^{no} (T, \bB_\rho)+\bA^2< \varepsilon$, $|q|\leq \varepsilon \rho$ and $q\in \Sing_f(T)$, then 
\begin{equation}\label{e:yet-decay}
\bE^{no} (T, \bB_r (q)) \leq C \frac{r^{\alpha}}{\rho^\alpha} \qquad \forall r \leq \frac{\rho}{2}\, .
\end{equation}
Assume now that $\bar q\in \overline{\Sing_f(T)}\cap \overline{\bB}_{\varepsilon \rho}$ and let $q_i\in \Sing_f(T)\cap \overline{\bB}_{\varepsilon \rho}$ be a sequence converging to it. Given that the constant $C$ is independent of $i$ we can pass in the limit in $i$ for the corresponding estimates of \eqref{e:yet-decay} and infer
\[
\bE^{no} (T, \bB_r (\bar q)) \leq C \frac{r^{\alpha}}{\rho^\alpha} \qquad \forall r < \frac{\rho}{2}\, .
\]
In particular $T$ has a unique flat tangent cone at $\bar q$. However, $\bar q$ cannot be a regular point, and thus, in particular, $\bar q\in \Sing_f(T)$. This shows that, upon choosing $\eta$ appropriately, we can assume that $\bS_f:= \overline{\bB}_\eta \cap {\rm Sing}_f (T)$ is a closed set. Fix now $k\in \mathbb N\setminus \{0,1\}$ and a point $q\in \bS_f (k)$. Since by Proposition \ref{p:almost-monot} the map
\[
\bS_f \ni q \mapsto \bI (q,0)
\]
is upper semicontinuous and it takes integer values, there is a closed ball $\overline{\bB}_r (q)$ with the property that $\bI (\cdot, 0) \leq k$ on $\overline{\bB}_r (q)$. Consider now a sequence $\{q_i\}\subset \overline{\bB}_r (q)\cap \bS_f (k)$ converging to some $\bar q$. From our considerations we know that $\bar q\in \bS_f$ and that $\bI (\bar q, 0)\leq k$. On the other hand, again the upper semicontinuity of the frequency implies $\bI (\bar q, 0) \geq k$. This shows that $\bar q \in \bS_f (k)\cap \overline{\bB}_r (q)$ and concludes the proof of Lemma \ref{l:closed}. \qed

\subsection{Proof of Proposition \ref{p:reduction_argument}} We examine the alternative (a) as it will become obvious that the same argument applies with the alternative (b). 

Consider the sequence $N_{\bar q,r_j}$ given by Proposition \ref{p:blow-up} and its limit $\bar N$. Let $\bS_j \subset T_{\bar q}\mathcal M$ denote the rescaled sets $r_j^{-1} \be (\bar q, \cdot)^{-1} (\bS_f (k) \cap \overline{\bB}_{r_j}(\bar q))$. Upon extraction of a (not relabeled) subsequence, we infer from Lemma \ref{l:closed} the existence of a compact set $\bar \bS$ which is the Hausdorff limit of $\bS_j$ and which, thanks to \eqref{e:lower-density} and the upper semicontinuity of the Hausdorff pre-measures $\mathcal{H}^s_\infty$ with respect to Hausdorff convergence of compacts sets, has positive $\mathcal{H}^{m-2+\delta}$ measure. Fix any point $q\in \bar \bS$, let $q_j \in \bS_j$ be a sequence converging to it and consider likewise the points $\bar q_j := \be (\bar q, r_j q_j)$. By Proposition \ref{p:almost-monot} and the assumption $\bI (\bar q_j, 0) = k$, we easily infer that, for any fixed radius $\rho$, 
\[
\liminf_{j\to\infty} \bI (\bar q_j, r_j \rho) \geq k\, .
\]
On the other hand the strong $W^{1,2}$ convergence of $N_{\bar q, r_j}$ to $\bar N$ immediately implies
\[
\frac{\rho \int |D\bar N|^2 \phi (\rho^{-1} |x-q|)}{- \int |x-q|^{-1} \phi' (\rho^{-1} |x-q|) |\bar N|^2} = \lim_{j\to\infty} \bI (\bar q_j, r_j \rho) \geq k
\]
Letting $\rho\downarrow 0$ we then conclude that $\bar N (q) = Q\a{0}$ and the frequency $I_{\bar N}(q,0)$ is at least $k$. However, by the $k$-homogeneity of $\bar N$, we necessarily have that the frequency does not exceed $k$. It thus turns out that the frequency of $\bar N$ at $q$ is $k$. By \cite[Section 10]{DLHMS_linear} (cf. \cite[Section 3.5]{DLS_Qvfr}), $\bar N (x+ \lambda q) = \bar N (x)$ holds for every $x\in \mathbb R^m$, every $\lambda \in \mathbb R$ and every $q\in \bar \bS$. Since the Hausdorff $(m-2+\delta)$-dimensional measure of $\bar \bS$ is positive, clearly $\bar \bS$ spans at least an $(m-1)$-dimensional vector space. If it were to span the whole $\mathbb R^m$, then $\bar N$ would be identically $Q\a{0}$, but we know from Proposition \ref{p:blow-up} that the latter is not possible. \qed

\subsection{Proof of Theorem \ref{t:even}} A function $\bar N$ as in Proposition \ref{p:reduction_argument} cannot exist. In fact, since the latter is nontrivial by Proposition \ref{p:blow-up}, the set of points where it takes the value $Q\a{0}$ coincides with the singular set of $\bar N$ and must necessarily be the hyperplane $V$, by \cite[Theorem 10.2]{DLHMS_linear}. But by Theorem \ref{thm:linear-what-is-needed} the frequency of $\bar N$ at $\mathcal{H}^{m-1}$-a.e. point $q\in V$ would have to be $1$, while we know from Proposition \ref{p:reduction_argument} that it must be an integer $k\geq 2$. 

This shows that:
\begin{itemize}
    \item When $m\geq 3$, $\mathbf{S}_f (k)$ has Hausdorff dimension at most $m-2$, which in turn implies that $\mathbf{S}_f$ has as well Hausdorff dimension at most $m-2$ and hence completes the proof of Theorem \ref{t:even} for the case $m\geq 3$;
    \item When $m=2$, $\mathbf{S}_f (k)$ is discrete, which in turn implies that $\mathbf{S}_f$ is countable and hence completes the proof of Theorem \ref{t:even} for the case $m=2$.
\end{itemize}
\qed

\subsection{Proof of Theorem \ref{t:even-structure}}
For $i=1,2,\dots$, let $\Lambda_i$ be the connected components of $(\Omega \cap \spt^p(T))\setminus \Sing(T)$. We first claim that if $\Lambda_i$ is nonorientable then $T\res \Lambda_i=Q\llbracket \Lambda_i\rrbracket\; \modp$. Indeed, the constancy lemma $\modp$ (see \cite[Theorem 7.6]{DPH_JMAA}) implies that there exists $\theta \in \left(-Q, Q\right] \cap \mathbb{Z}$ such that $T \res \Lambda_i = \theta\,\a{\Lambda_i}\;\modp$. Assume by contradiction that $\theta\neq Q$, and exploit the fact that $\Lambda_i$ is nonorientable to find a loop $\gamma$ on $\Lambda_i$ and a number $\delta >0$ with the following property: the $\delta$-tubular neighborhood 
\[
B_\delta(\gamma):=\{q\in \Lambda_i\,\colon\,{\rm{dist}}(q,\gamma)<\delta\}
\]
is nonorientable, and there exists an $(m-1)$-dimensional orientable surface $B\subset B_\delta(\gamma)$ such that $B_\delta(\gamma)\setminus B$ is orientable. Then, for a suitable choice of orientation on $B$, one has that $(\partial (T\res B_\delta(\gamma)))\res B_\delta(\gamma)= 2\theta\llbracket B\rrbracket\neq 0\; \modp$, which is a contradiction. 

Next, we set
\begin{eqnarray*}
T_o := & T \res \bigcup \left\lbrace \Lambda_i \, \colon \, \mbox{$\Lambda_i$ is orientable} \right\rbrace\,, \\
T_n := &T \res \bigcup \left\lbrace \Lambda_i \, \colon \, \mbox{$\Lambda_i$ is nonorientable} \right\rbrace\,, 
\end{eqnarray*}
and we proceed with the proof of the conclusions of Theorem \ref{t:even-structure}. 

We first claim that $\spt^p(T_o) \cap \spt^p (T_n) \subset \mathcal S$: by the end of the proof we will then upgrade this conclusion to the stronger $\spt^p(T_o) \cap \spt^p (T_n) = \emptyset$, which is the only thing to check for (i). The fact that $\spt^p(T_o) \cap \spt^p (T_n) \subset \mathcal S$ follows from Theorem \ref{t:even} and the claim above on the multiplicity of $T$ on the nonorientable components, after recalling from Definition \ref{def:free-boundary} that the multiplicities of sheets concurring at points of $\sing(T) \setminus \mathcal S$ are integers $k < Q$.

Next, we prove (iii). We set $T_2:= Q^{-1}\,T_n$, and we notice that $T_2$ is a representiative $\moddue$ as a consequence of the above claim on the multiplicities of $T_n$. The boundary $\moddue$ $\partial^2[T_2]$ is then a flat $(m-1)$-chain $\moddue$ with $\spt^2(\partial^2[T_2])\subset \mathcal S$. Since $\Ha^{m-1}(\mathcal S)=0$, \cite{White_def} implies that $\partial^2[T_2]=0$.

We proceed now with the proof of (iv). The second sentence is a consequence of the first sentence and of \cite{NV}. The first sentence is a well-kown fact. A quick proof can be achieved as follows. Having fixed $q\in \Sigma \cap \Omega$, choose $\rho$ sufficiently small so that $\bB_\rho (q) \cap \Sigma$ is diffeomorphic to the $m+1$-dimensional ball. Since the $\mathbb Z_2$-homology of the latter is trivial and $\partial T_2 \res \bB_\rho (q) = 0\, \moddue$, we conclude that $T_2$ is a boundary $\moddue$, i.e. there is a flat chain $S$ $\moddue$ of dimension $m+1$ supported in $\Sigma$ such that $(T_2 - \partial S ) \res \bB_\rho (q) = 0\, \moddue$. Since the dimension of $S$ equals the dimension of $\Sigma$, the latter has a representative which is a Caccioppoli set of $\Sigma \cap \bB_\rho (q)$. Its boundary is the desired integral current representing $T_2 \res \bB_\rho (q)$ and it follows that it has to be area minimizing. 

Next we conclude the proof of (i), showing that in fact $\spt^p(T_o) \cap \spt^p (T_n) = \emptyset$. Fix a point $q\in \spt^p (T_n) = \spt^2 (T_2)$ and assume by contradiction that it belongs to $\spt^p (T_o)$ as well. Consider a ball $\bB_\rho (q)\subset \subset \Omega$ and let $\Gamma := (\bB_\rho (q) \cap \Sigma)\setminus \spt^2 (T_2)$: the latter is an open subset of $\Sigma \cap \bB_\rho (q)$. Enumerate its connected components by $\{\Gamma_i\}$ and recall that they might be, in principle, infinitely many. Consider now the current $T_o^i := T_o \res \Gamma_i$. Since we know that $\spt^p(T_o) \cap \spt^2(T_2) \subset \mathcal S$, the boundary $\modp$ of $T_o^i$ is supported in $\partial \bB_\rho (q)$. 
Some of these currents must be nonzero, because the regular part of $T_o$ does not intersect the support of $T_2$: in fact, since $q\in \spt^p (T_o)$ there must be a sequence of points $\{q_j\}\subset \Gamma \cap \spt^p (T_o)$ such that $q_j \to q$. Each $q_j$ will belong to the support of some $T_o^{i(j)}$. 

We first claim that $\{i(j)\}$ is bounded. Note that each such $T_o^i$ has empty boundary $\modp$ and it is area minimizing $\modp$ in $\bB_\rho (q) \cap \Sigma$. In particular, the presence of a point in $\spt^p (T_o^i)\cap \bB_{\rho/2} (q)\neq \emptyset$ would imply, by the monotonicity formula, that 
\[
\|T_o^i\| (\bB_\rho (q)) \geq c_0 >0
\]
for some positive constant $c_0$. Since 
\[
\|T_o\| (\bB_\rho (q)) = \sum_i \|T_o^i\| (\bB_\rho (q))\, ,
\]
the number of distinct $T_o^i$ whose supports intersect $\bB_{\rho/2} (q)$ is finite.

We thus conclude that there is one $T_o^i$ which is nontrivial and with the property that $q$ belongs to its support. This puts us in the position of applying the strong maximum principle of \cite{Wic2} to the varifolds induced by $T_2$ and $T_o^i$. But then we would conclude that the support of $T_o^i$ is contained in the support of $T_2$, which is a contradiction, because it would imply that the regular parts of $T_n$ and $T_o$ have nonempty intersection.

Finally, the proof of (ii) is analogous to that of \cite[Corollary 1.10]{DLHMS}, and thus we omit it.
\qed



\bibliographystyle{plain}
\bibliography{Biblio}

\end{document}